\documentclass[11pt]{article}

\usepackage{amssymb}
\usepackage{amsmath}
\usepackage{amsthm}
\usepackage{color}
\usepackage{tikz}
\usepackage{accents}

\usepackage{graphics,graphicx}
\usepackage{tikz-cd}

\def\rr{{\mathbb R}}
\def\rn{{{\rr}^d}}

\def\nn{{\mathbb N}}

\def\fz{\infty}
\def\az{\alpha}

\def\vz{\varphi}
\def\tz{\Theta}

\def\wz{\widetilde}

\def\rr{{\mathbb R}}

\def\nn{{\mathbb N}}

\def\fz{\infty}
\def\az{\alpha}

\def\vz{\varphi}
\def\tz{\Theta}

\def\wz{\widetilde}

\def\laz{\langle}
\def\raz{\rangle}

\def\r{\right}
\def\lf{\left}

\newcommand{\re}{\mathbb{R}}\newcommand{\N}{\mathbb{N}}
\newcommand{\C}{\mathbb{C}}

\newcommand{\R}{{\re}^d}

\newcommand{\cs}{{\mathcal S}}
\newcommand{\cd}{{\mathcal D}}
\newcommand{\cl}{{\mathcal L}}
\newcommand{\cm}{{\mathcal M}}

\newcommand{\ce}{{\mathcal E}}

\newcommand{\cn}{{\mathcal N}}
\newcommand{\cf}{{\mathcal F}}
\newcommand{\cfi}{{\cf}^{-1}}

\newcommand{\supp}{{\rm supp \, }}

\newcommand{\ls}{\lesssim}

\newcommand{\be}{\begin{equation}}
\newcommand{\ee}{\end{equation}}
\newcommand{\beq}{\begin{eqnarray}}
\newcommand{\beqq}{\begin{eqnarray*}}
\newcommand{\eeq}{\end{eqnarray}}
\newcommand{\eeqq}{\end{eqnarray*}}

\numberwithin{equation}{section}

\newtheorem{satz}{Theorem}[section]
\newtheorem{defi}[satz]{Definition}
\newtheorem{cor}[satz]{Corollary}\newtheorem{lem}[satz]{Lemma}
\newtheorem{prop}[satz]{Proposition}
\newtheorem{rem}[satz]{Remark}

\begin{document}

\title{Complex Interpolation of Lizorkin-Triebel-Morrey Spaces on Domains}

\date{\today}

\author{
Ciqiang Zhuo \thanks{Corresponding author;
Key Laboratory of Computing and Stochastic Mathematics
(Ministry of Education), School of Mathematics and Statistics, Hunan Normal University, Changsha, Hunan 410081, People's Republic of China;
E-mail: cqzhuo87@hunnu.edu.cn}
\and Marc Hovemann
\thanks{Friedrich-Schiller-University Jena, Ernst-Abbe-Platz 2, 07737 Jena, Germany;
 E-mail: marc.hovemann@uni-jena.de}
\and Winfried Sickel \thanks{
 Friedrich-Schiller-University Jena, Ernst-Abbe-Platz 2, 07737 Jena, Germany;
 E-mail: winfried.sickel@uni-jena.de}}

\maketitle

\begin{abstract}
 In this article the authors study complex  interpolation of Sobolev-Morrey spaces
and their generalizations, Lizorkin-Triebel-Morrey  spaces.
Both scales are considered on bounded domains.
Under certain conditions on the parameters
the outcome belongs to the scale of the  so-called diamond spaces.
\end{abstract}


MSC subject class: {46B70,  46E35}


\section{Introduction and Main Results}\label{s1}


One of the most popular formulas in interpolation theory is given by

\be
\label{w-001}
[L_{p_0}(\R),L_{p_1}(\R)]_\tz = L_p(\R)\, ,
\ee
where $1\le p_0 <p_1 \le \infty$, $0 < \Theta < 1$ and
$\frac1p:=\frac{1-\tz}{p_0}+\frac{\tz}{p_1}$.
Here $[X_0,X_1]_\Theta$ denotes Calderon's first complex interpolation method or just the complex method.
Morrey spaces $\cm_p^u (\R)$ are generalizations of the Lebesgue spaces in view of $\cm_p^p (\R) = L_p (\R)$.
Within the larger family of Morrey spaces the formula \eqref{w-001} is a singular point. Essentially as a result of
Lemari{\'e}-Rieusset \cite{LR}, \cite{LR2}
it is known that

\be
\label{w-002}
[\cm^{u_0}_{p_0}(\R),\cm^{u_1}_{p_1}(\R)]_\tz \neq \cm_p^u(\R)\, ,
\ee
except the trivial cases given by either
$u_0 =p_0$, $u_1=p_1$, i.e., the Lebesgue case, or $u_0 =u_1$, $p_0=p_1$.
In \cite{ysy3} and \cite{hns17} different explicit descriptions of the spaces
$[\cm^{u_0}_{p_0}(\R),\cm^{u_1}_{p_1}(\R)]_\tz $ can be found.
The characterization given  in \cite{hns17} is the preferable one. When switching from Lebesgue spaces to Morrey spaces we add two phenomena, one local and one global, see Definition \ref{Morrey1} below.
Hence, when turning to spaces defined on bounded domains, the situation is becoming more easy, because the global condition
plays no role anymore.
Based on this observation, in \cite{ysy3} one can find the formula

\be \label{w-003}
[\cm^{u_0}_{p_0}([0,1]^d),\cm^{u_1}_{p_1}([0,1]^d)]_\tz = \accentset{\diamond}\cm_p^u([0,1]^d)\, ,
\ee
if
\[
1\le p_0 < u_0 <\infty, \quad 1 < p_1< u_1 <\infty, \quad p_0 < p_1,
\quad 0 < \Theta < 1
\]
and
\[
p_0 u_1 = p_1 u_0\, , \quad
\frac1p:=\frac{1-\tz}{p_0}+\frac{\tz}{p_1}\, , \quad
\frac1u:=\frac{1-\tz}{u_0}+\frac{\tz}{u_1}\, .
\]
For a domain $ \Omega \subset \R  $ the space $\accentset{\diamond}\cm_p^u(\Omega)$ is defined as the closure of the smooth functions with respect to
the norm of the space $\cm_p^u(\Omega)$. The aim of this paper will consist in an extension of
\eqref{w-003} to smoothness spaces built on Morrey spaces, namely Lizorkin-Triebel-Morrey spaces $\ce^s_{u,p,q}(\Omega)$, where
$\Omega \subset \R$ is a bounded Lipschitz domain.
For doing that we will only investigate cases where the Lemari{\'e}-Rieusset condition
$p_0 u_1 = p_1 u_0$ is satisfied.
Our main result reads as follows.

\begin{satz}\label{main1}
Let $\Omega \subset \R$ be either  a bounded Lipschitz domain
if $d\ge 2$ or a bounded interval if $d=1$.
Under the following conditions on the parameters
\begin{itemize}
 \item[(a)] $1\le p_0< p_1<\infty$, $p_0 \le u_0<\infty$, $p_1 \le  u_1 <\infty$;
 \item[(b)] $1\le q_0\, ,  q_1 \le \infty$, $\min (q_0,q_1)<\infty$;
 \item[(c)] $p_0 \, u_1 = p_1 \, u_0$;
 \item[(d)] $s_0,s_1\ge 0$; either $s_0 < s_1$ or $0 < s_0 = s_1$ and $q_1\le q_0$;
 \item[(e)] $0 < \Theta <1$,  $ \frac1p:=\frac{1-\tz}{p_0}+\frac{\tz}{p_1}$,
$\frac1u:=\frac{1-\tz}{u_0}+\frac{\tz}{u_1}$,
$\frac1q:=\frac{1-\tz}{q_0}+\frac{\tz}{q_1}$ \\
and\quad
$s:= (1-\tz)s_0 + \tz s_1$;
\end{itemize}
it holds
\be
\label{w-004}
[\ce^{s_0}_{u_0,p_0,q_0}(\Omega),\ce^{s_1}_{u_1,p_1,q_1}(\Omega)]_\tz = \accentset{\diamond}\ce_{u,p,q}^s(\Omega)\, .
\ee
\end{satz}

Lizorkin-Triebel-Morrey spaces $\ce^s_{u,p,q}(\Omega)$
are generalizations of Lizorkin-Triebel spaces $F^s_{p,q}(\Omega)$,
more exactly, if $u=p$ we have $F^s_{p,q}(\Omega) = \ce^s_{p,p,q}(\Omega) $. Hence we get back the well-known formula
\be
\label{w-005}
[F^{s_0}_{p_0,q_0}(\Omega),F^{s_1}_{p_1,q_1}(\Omega)]_\tz =
\accentset{\diamond}{F}^{s}_{p,q} (\Omega) =
F^s_{p,q}(\Omega)\, ,
\ee
but under the extra condition (d). The Lemari{\'e}-Rieusset condition
(c) disappears in this case. There is a certain list of references for \eqref{w-005}. Let us mention at least
Triebel \cite[Thm.~2.4.2.1]{t78} ($\Omega = \R$ or a bounded $C^\infty$ domain), Frazier, Jawerth \cite{fj90} ($\Omega = \R$),
Kalton, Mayboroda, Mitrea \cite{kmm} ($\Omega = \R$) and Triebel \cite{t02} (bounded Lipschitz domains). There is an interesting special case, given by the  Sobolev-Morrey spaces, see Section \ref{LTMSS}
and Lemma \ref{LP}.

\begin{cor}\label{MSM}
Let $0 < \tz <1$, $m_0 \in \N_0, m_1\in\N $, and either  $m_0 < m_1$ or $0< m_0\le m_1$.
Let $1 < p_0 < p_1 <\infty$, $p_0 < u_0< \infty$, $p_1 < u_1 < \infty$ and $p_0\, u_1=p_1\, u_0$.
We define
\[
s:=(1-\tz)m_0 + \tz m_1\, , \quad
\frac1p :=\frac{1-\tz}{p_0}+\frac\tz{p_1}\quad
\mbox{and} \quad \frac1u :=\frac{1-\tz}{u_0}+\frac{\tz}{u_1}\, .
\]
Let $\Omega \subset \R$ be either  a bounded Lipschitz domain
if $d\ge 2$ or a bounded interval if $d=1$.
Then we have
\be\label{ws-09}
\big[ W^{m_0}\cm^{u_0}_{p_0}(\Omega), W^{m_1}\cm^{u_1}_{p_1}(\Omega) \big]_\Theta = \accentset{\diamond}\ce_{u,p,2}^{s}(\Omega) .
\ee
In particular, if $s = m \in\N$, then
 \be\label{ws-09b}
\big[ W^{m_0}\cm^{u_0}_{p_0}(\Omega), W^{m_1}\cm^{u_1}_{p_1}(\Omega) \big]_\Theta = \accentset{\diamond}W^m \cm_{p}^u (\Omega)
\ee
follows.
\end{cor}

\noindent
There is another situation in which one can calculate
$[\ce^{s_0}_{u_0,p_0,q_0}(\Omega),\ce^{s_1}_{u_1,p_1,q_1}(\Omega)]_\tz$ .

\begin{satz}\label{main1b}
Let $\Omega \subset \R$ be as in Theorem \ref{main1}.
Let the parameters satisfy the conditions (a), (b), (c) and (e).
In addition we require
\be \label{wss-20}
\mbox{(d')}\qquad  s_0, s_1\in \re \quad \mbox{and}\quad s_0- \frac{d}{u_0} > s_1 - \frac {d}{u_1}\, .
\ee
Then \eqref{w-004} holds as well.
\end{satz}

Clearly, in Theorem \ref{main1b} we always have $s_0 >s_1$. So there is no overlap
with Theorem \ref{main1}.
For convenience of the reader we add the consequences for the interpolation of
Sobolev-Morrey spaces.

\begin{cor}\label{MSM2}
Let $0 < \tz <1$, $1 < p_0 < p_1 <\infty$, $p_0 < u_0<\infty$, $p_1 < u_1< \infty $ and $p_0\, u_1=p_1\, u_0$.
Let
$m_0 \in \N, m_1\in\N_0 $ and $m_0 -\frac{d}{u_0} > m_1 - \frac{d}{u_1}$.
We define
\[
s:=(1-\tz)m_0 + \tz m_1\, , \quad
\frac1p :=\frac{1-\tz}{p_0}+\frac\tz{p_1}\quad
\mbox{and} \quad \frac1u :=\frac{1-\tz}{u_0}+\frac{\tz}{u_1}\, .
\]
Let $\Omega \subset \R$ be either  a bounded Lipschitz domain
if $d\ge 2$ or a bounded interval if $d=1$.
Then \eqref{ws-09} holds.
In particular, if $s = m \in\N$, then also \eqref{ws-09b} is true.
\end{cor}

The formula \eqref{w-004} does not hold in general. There are many counterexamples.

\begin{prop}\label{main3}
Let $\Omega \subset \R$ be a domain.
We assume that
\begin{itemize}
 \item[(a)] $1\le p_0< p_1<\infty$, $p_0 < u_0<\infty$, $p_1 < u_1<\infty$;
 \item[(b)] $1\le q_0\, ,  q_1 \le \infty$;
 \item[(c)] $p_0 \, u_1 = p_1 \, u_0$.
\end{itemize}
If $0 < s_0 < d/u_0$ and  if
\be\label{wss-18}
s_1:= s_0 -d \, \Big(\frac 1{u_0} - \frac 1{u_1}\Big) >0\, ,
\ee
then with
$0 < \Theta <1$, $ \frac1p:=\frac{1-\tz}{p_0}+\frac{\tz}{p_1}$,
$\frac1u:=\frac{1-\tz}{u_0}+\frac{\tz}{u_1}$,
$\frac1q:=\frac{1-\tz}{q_0}+\frac{\tz}{q_1}$
and \\
$s:= (1-\tz)s_0 + \tz s_1$ it holds
\be
\label{w-006}
[\ce^{s_0}_{u_0,p_0,q_0}(\Omega),\ce^{s_1}_{u_1,p_1,q_1}(\Omega)]_\tz \not\subset \accentset{\diamond}\ce_{u,p,q}^s(\Omega)\, .
\ee
\end{prop}

\noindent
Finally, we add a few comments concerning the situation on $\R$.

\begin{itemize}
  \item
The conditions $s_0,s_1 \in \re$, $1\le q_0,q_1\le \infty$
together with
(a),(c),(e) from Theorem \ref{main1} guarantee the continuous embedding
\be
\label{w-009}
[\ce^{s_{0}}_{u_{0},p_{0},q_0}(\R),\ce^{s_{1}}_{u_{1},p_{1},q_1}(\R)]_\tz \hookrightarrow \ce_{u,p,q}^s(\R)\, .
\ee
We refer to  Yang, Yuan, Zhuo \cite{yyz}.
\item
 If  $1<  p\le u <  \infty$, $1< q_0 < q_1 \le \infty$, $s\in\re$,
 $0 < \Theta <1$ and $\frac1q:=\frac{1-\tz}{q_0}+\frac{\tz}{q_1}$, then
\be
\label{w-007}
[\ce^{s}_{u,p,q_0}(\R),\ce^{s}_{u,p,q_1}(\R)]_\tz = \ce_{u,p,q}^s(\R)
\ee
holds. We refer to Sawano and Tanaka \cite{sat}.
\end{itemize}

We supplement these assertions by one negative and one positive result.

\begin{prop}\label{main4}
{\rm (i)}
Let $s_0$ and $s_1$ be positive real numbers.
Let the conditions (a), (b), (c) and (e)  from Theorem \ref{main1} be satisfied.
Then
\[
\accentset{\diamond}\ce_{u,p,q}^{s}(\R) \not \subset [\ce_{u_0,p_0,q_0}^{s_0}(\R), \ce_{u_1,p_1,q_1}^{s_1}(\R)]_\Theta  \,  .
\]
{\rm (ii)} Under the same restrictions as in Theorem \ref{main1b}
we have
\[
[\ce_{u_0,p_0,q_0}^{s_0}(\R), \ce_{u_1,p_1,q_1}^{s_1}(\R)]_\Theta
\hookrightarrow \accentset{\diamond}\ce_{u,p,q}^{s}(\R) \, .
\]
\end{prop}
This supplements the knowledge about Morrey spaces since it holds
\[
 \accentset{\diamond}\cm_{p}^{u}(\R) \not \subset
 [\cm_{p_0}^{u_0}(\R), \cm_{p_1}^{u_1}(\R)]_\Theta \hookrightarrow
 \accentset{\diamond}\cm_{p}^{u}(\R) \,  ,
\]
 if $ 1\le p_0 <p_1 < \infty$, $p_0 < u_0<\infty$,  $p_1 < u_1<\infty$ and  $p_0 \, u_1 = p_1 \, u_0$,
 see \cite[Cor.~2.38]{ysy3}. So all in all the general picture concerning complex interpolation of Lizorkin-Triebel-Morrey spaces seems to be more complicated than expected.
Below we have tried to make the situation on domains a bit more transparent. We shall plot an $(1/u, s)$ diagram. The influence of the parameters $p_0,q_0,p_1,q_1$ is ignored. First we fix a point $(1/u_0,s_0)$. Then we have indicated for which regions in the plane we may apply
either Theorem \ref{main1} or Theorem \ref{main1b} or Proposition \ref{main3}.


\begin{tikzpicture}[thick]
\draw[->] (-2,0) -- (10,0) ;
\draw[->] (-2,-2) -- (-2,6) ;
\draw[dotted] (-2,0) -- (8,6) ;
\draw (7,-0.1) -- (7,0.1) ;
\draw (9,-0.1) -- (9,0.1) ;
\draw (-2.1,4) -- (-1.9,4) ;
\draw (7,4) circle (0.1cm) ;
\draw (7,4) -- (7,6) ;

\draw (-2,-1.4) -- (7,4) ;

\draw (-2,4) -- (7,4) ;

\draw[dotted] (7,0) -- (7,4) ;


\draw (-2.1,-1.4) -- (-1.9,-1.4) ;


\node at (-2.3,-1.4) {$t$} ;
\node at (-2.3,5.8) {$s$} ;
\node at (8.8,6) {$s= d/u$} ;
\node at (-2.3,4) {$s_0$} ;
\node at (7,-0.5) {$1/u_0$} ;
\node at (9,-0.5) {$1$} ;
\node at (-1.8,-0.3) {$0$} ;
\node at (10,-0.5) {$1/u$} ;
\node at (2,5) {Theorem 1.1} ;
\node at (0,2.5) {open} ;
\node at (4.5,0.8) {Theorem 1.3} ;
\end{tikzpicture}
{~}\\
The point $t$ is given by the Sobolev-type embedding as $t:= s_0 -d/u_0$.
In the open rectangle $\{ (1/u_1, s_1):~~ u_0 < u_1\, , ~~s_1 >s_0\}  $
we can apply Theorem \ref{main1}.
In the open triangle with corner points $(0,t)$, $(1/u_0, s_0)$, $(0,s_0)$
we do not know
$[\ce_{u_0,p_0,q_0}^{s_0}(\Omega), \ce_{u_1,p_1,q_1}^{s_1}(\Omega)]_\Theta$.
Below of the line connecting $(0,t)$ and $(1/u_0,s_0)$ we may apply
Theorem \ref{main1b}.
On this critical line Proposition \ref{main3} applies.


{~}\\
This article is organized as follows.
In Section \ref{s-Morrey} we recall the
definitions of  Morrey spaces and Lizorkin-Triebel-Morrey spaces on
the Euclidean space $\R$ as well as on domains. In addition we introduce the diamond spaces. Moreover, a few basic properties of these classes are recalled as well.

Section \ref{s-diamon} is the most important one within this paper.
We investigate the spaces $\accentset{\diamond}{\ce}^{s}_{u,p,q} (\R)$ in detail. In the Lemmas \ref{lem1}
and \ref{lem2} we characterize this space via differences, which is very important
for us. We use this characterization to prove an embedding property
on the intersection of Lizorkin-Triebel-Morrey spaces in Lemma \ref{step2} below.

Section \ref{slava} is devoted to the existence of an universal  bounded  linear extension operator which maps
${\ce}^{s}_{u,p,q} (\Omega)$ into ${\ce}^{s}_{u,p,q} (\rn)$.
Here we employ Rychkov's method and construction.

Interpolation  will be the main topic
in Section \ref{s-interpolation}. Our treatment of the complex
interpolation of Lizorkin-Triebel-Morrey spaces will be reduced
to the calculation of a closure of some intersections by means of a formula due to Shestakov \cite{s74}, \cite{s74m}.

For convenience of the reader, in Section \ref{overview} we will give
a short  overview about interpolation of smoothness Morrey spaces.

Finally, in Section \ref{Ende} a number of open problems
is collected.

\vspace{0,3 cm}

But at first we want to fix some notation.


\subsection*{Notation}


For any $x\in \R$ and $r\in(0,\fz)$ we use $B(x,r)$ to denote the ball in $\R$ centered at $x$ with radius $r$,
namely, $B(x,r):=\{y\in\R:\ |x-y|<r \}$.
If $\alpha = (\alpha_1, \ldots \, ,\alpha_d) \in \N_0^d$ and $f:~\Omega \to \C$, then we put
\[
 D^\alpha f (x) = \frac{\partial^{|\alpha|} f}{\partial x_1^{\alpha_1}\, \ldots \, x_d^{\alpha_d}}(x)\, , \qquad x \in \Omega\, .
\]
For a domain $\Omega \subset \R$ we define
$\cd (\Omega)$  as the set of all functions $f$ having derivatives
up to any order
and fulfill $\supp f\subset \Omega$. $\cd' (\Omega)$ is the dual space of $\cd(\Omega)$.
The symbol $\cl (X \to Y)$ denotes the set of all linear bounded operators
from $X$ to $Y$.
By $C^\infty (\R)$
we denote the collection of all complex-valued infinitely differentiable functions
on $\R$, by $C^\infty_0 (\R)$ the subset consisting of those elements having
compact support.
Let $\cs (\R)$ denote the Schwartz space of all complex-valued, rapidly decreasing  and infinitely differentiable functions on $\R$.
By $\cs'(\R)$ we denote the collection of all tempered distributions on
$\R$, i.e., the topological dual of $\cs(\R)$, equipped with
the  weak-$*$ topology.
The symbol $\cf$ refers to  the Fourier transform, $\cfi$ to its inverse transformation,
both defined on $\cs'(\R)$.
All function spaces which  we consider in this paper are subspaces of $\cs'(\R)$,
i.e. spaces of equivalence classes w.r.t. almost everywhere equality.
However, if such an equivalence class  contains a continuous representative,
then usually we work with this representative and call also the equivalence class a continuous function. The symbols  $C, C_1, c, c_{1} \ldots $ denote  positive constants that depend only on the fixed parameters $d,s,u,p,q$ and probably on auxiliary functions. Unless otherwise stated their values may vary from line to line. Sometimes we also use the symbol $ \ls $
instead of $ \le $. The meaning of $A \ls B$ is
given by: there exists a positive constant $C$ such that
 $A \le C \,B$.
Mainly in Section \ref{slava} we will use the abbreviation (with modification if $ q = \infty $)
\[
\|\, \{f_j\}_{j=0}^\infty\, |\cm_p^u (\ell^s_q (\R))\|:=
\Big \| \Big ( \sum_{j=0}^\infty |2^{js}\, f_j(\, \cdot \, )|^q   \Big )^{\frac{1}{q}} \Big| \cm_p^u(\rn)  \Big \| .
\]


\section{Smoothness Morrey spaces\label{s-Morrey}}


In this section we recall the definitions of the function spaces under consideration.


\subsection{Morrey spaces}


Morrey spaces can be understood as a replacement
(or a generalization) of the Lebesgue spaces $L_p (\rr^d)$.
This is immediate in view of their definition.

\begin{defi}\label{Morrey1}
 Let $1 \le p\le u < \infty$.
Then the Morrey space $\mathcal{M}^u_{p}(\R)$
is defined as  the collection of all locally Lebesgue-integrable functions $f$ on $\R$ such that
\begin{equation}\label{morrey7}
\|f \vert \mathcal{M}^u_{p}(\R) \| :=  \sup_{B}
|B|^{\frac{1}{u}- \frac{1}{p}}\lf[\int_B |f(x)|^p\,dx\r]^{\frac{1}{p}}<\infty\, ,
\end{equation}
where the supremum is taken over all balls $B$ in $\R$.
\end{defi}

Clearly, there is a big difference between the cases $|B|>1$
and $|B|\le 1$.
We have a strong local condition combined with a weak global condition. Later we shall need some knowledge about certain subspaces of Morrey spaces. Therefore we give the following definition.

\begin{defi}\label{d-cm}
Let $X$ be a Banach space of distributions or functions.

{\rm (i)} By
$\accentset{\diamond}{X}$ we denote the closure in $X$ of the set of all infinitely often
differentiable functions $f$ that fulfill
$D^\alpha f \in X$ for all $\alpha \in \N_0^d$.

{\rm (ii)}
Let $C_0^\infty (\R) \hookrightarrow X$.
Then by $\mathring{X}$ we denote the closure of $C_0^\infty (\R)$ in $X$.
\end{defi}

The next lemma gives explicit descriptions
of  $\mathring{\cm}^{u}_{p}(\R)$ and
$\accentset{\diamond}{\cm}^{u}_{p}(\R)$, very much in the spirit of the original definition of  Morrey spaces, see \cite[Lemma 2.33]{ysy3}.

\begin{lem}\label{morrey43}
Let $1 \le p < u < \infty$.
\\
{\rm (i)}  $\mathring{\cm}^{u}_{p}(\R)$ is equal to the collection of all $f \in {\cm}^{u}_{p}(\R)$ having the
following properties:
\begin{eqnarray}
&& \qquad \lim_{r \downarrow 0}
|B(y,r)|^{\frac{1}{u}- \frac{1}{p}}\lf[\int_{B(y,r)} |f(x)|^p\,dx\r]^{\frac{1}{p}}   =  0 \, ,
\label{ws-51}
\\
&&
\nonumber
\\
\label{ws-50}
&& \qquad
\lim_{r \to \infty}
|B(y,r)|^{\frac{1}{u}- \frac{1}{p}}\lf[\int_{B(y,r)} |f(x)|^p\,dx\r]^{\frac{1}{p}}  =  0\, ,
\end{eqnarray}
both uniformly in  $y \in \R $, and
\begin{equation}
 \label{ws-74}
 \lim_{|y|\to \infty}
|B(y,r)|^{\frac{1}{u}- \frac{1}{p}}\lf[\int_{B(y,r)} |f(x)|^p\,dx\r]^{\frac{1}{p}}  =  0
\end{equation}
uniformly in $r\in(0,\fz)$.
\\
{\rm (ii)} $\accentset{\diamond}{\cm}^{u}_{p}(\rn)$ is equal to the collection of all $f \in {\cm}^{u}_{p}(\rn)$
such that \eqref{ws-51} holds true
uniformly in  $y \in \R$.
\end{lem}


\subsection{Lizorkin-Triebel-Morrey spaces  on $\R$}
\label{LTMSS}


In what follows we will define the Lizorkin-Triebel-Morrey spaces $ \ce^{s}_{u,p,q}(\re^d)   $. For that purpose we need some additional notation. Let $\varphi_0 \in C_0^{\infty}({\R})$ be a non-negative function such that
 $\varphi_0(x) = 1$ if $|x|\leq 1$ and $ \varphi_0 (x) =0$ if $|x|\geq 3/2$.
For $k\in \N$ we define
\beqq
         \varphi_k(x) = \varphi_0(2^{-k}x)-\varphi_0(2^{-k+1}x) ,\qquad\ x \in \R\, .
\eeqq
Because of
\[
 \sum_{k=0}^\infty \varphi_k(x) = 1\, , \qquad x\in \R\, ,
\]
and
\[
 \supp \varphi_k \subset \big\{x\in \R: \: 2^{k-1}\le |x|\le 3 \cdot 2^{k-1}\big\}\, , \qquad k \in \N\, ,
\]
we  call the system $(\varphi_k)_{k\in \N_0 }$ a smooth  dyadic decomposition of unity on $\R$.
Clearly, by the Paley-Wiener-Schwarz theorem, $\cfi[\varphi_{k}\, \cf f]$ is a smooth function for all
$f\in \cs'(\R)$.

\begin{defi}\label{LTMS}
Let $1 \le p \le u< \infty $, $1\le q \le \infty$ and $s\in \re$.
Let $(\varphi_k)_{k\in \N_0 }$ be the above system.
Then the Lizorkin-Triebel-Morrey space  $ \ce^{s}_{u,p,q}(\re^d)$ is the
         collection of all tempered distributions $f \in \mathcal{S}'(\R)$
         such that
\beqq
          \|\, f \, |\ce^s_{u,p,q}(\R)\|_{\varphi_0} :=
         \bigg\| \Big ( \sum\limits_{k=0}^{\infty} 2^{k s q}\, |\cfi[\varphi_{k}  \cf f](\, \cdot \, )|^q\Big)^{\frac{1}{q}}
                  \bigg| \cm^u_p (\re^d)\bigg\| <\infty
\eeqq
(with usual modification if $q= \infty$).
\end{defi}

\begin{rem}
 \rm
The spaces  $ \ce^{s}_{u,p,q}(\re^d)$ are Banach spaces.
They do not depend on the chosen generator $\varphi_0$ of the smooth dyadic decomposition of unity
in the sense of equivalent norms. We refer, e.g., to \cite{ysy} or \cite{t14}.
For this reason we will drop the dependence on $\varphi_0$ in notations and simply write
$ \|\, \cdot \, |\ce^s_{u,p,q}(\R)\|$.
\end{rem}
Now we want to collect some basic properties of the spaces $ \ce^{s}_{u,p,q}(\re^d)   $. Most of them will be used later. At first we recall a characterization of $\ce^s_{u,p,q}(\R)$, due to Tang and Xu \cite{tx}, in terms of lower order derivatives.

\begin{lem}\label{lift}
 Let $m \in \N$, $s \in \re $, $1\le p \le u <  \infty$ and   $1 \leq q \le \infty$.
 Then we have $f \in \ce^s_{u,p,q}(\R) $ if, and only if, the tempered distribution $f$ and its
distributional derivatives $\frac{\partial^m f}{\partial x^m_j}$, $j=1, \ldots \, , d$, belong to
$\ce^{s-m}_{u,p,q}(\R)$.
Furthermore, the norms $\|\, f\, |{\ce^s_{u,p,q}(\R)}\|$ and
\[
\|\, f \, |\ce^{s-m}_{u,p,q}(\R)\| +  \sum_{j=1}^d \Big\|\frac{\partial^m f}{\partial x^m_j}
\, \Big| \ce^{s-m}_{u,p,q}(\R)\Big\|
\]
are equivalent.
\end{lem}

\noindent
The classical forerunner of Lemma \ref{lift} can be found in \cite[Thm.~2.3.8]{t83}.

\begin{lem}\label{lift2}
Let  $s \in \re $, $1\le p \le u <  \infty$ and   $1 \leq q \le \infty$.
If  $f \in \ce^s_{u,p,q}(\R) $, then $D^\alpha f \in \ce^{s-|\alpha|}_{u,p,q}(\R)$ for all $\alpha \in \N_0^d$.
Furthermore, there exists a constant  $c_\alpha $ such that
\[
\|\, D^\alpha f \, |\ce^{s-|\alpha|}_{u,p,q}(\R)\| \le c_\alpha \,
\| \, f \, | \ce^{s}_{u,p,q}(\R)\|
\]
holds for all $f\in \ce^s_{u,p,q}(\R)$.
\end{lem}

\noindent
In case $u=p$ this can be found in \cite[Thm.~2.3.8]{t83}.
The generalization to $p\neq u$ can be done in the same way as the proof of Lemma \ref{lift}.
\\
Lizorkin-Triebel-Morrey spaces  are generalizations of Sobolev-Morrey spaces.

\begin{defi}
 Let $m \in \N$ and  $1\le p \le u< \infty$.
Then the Sobolev-Morrey space $W^m\cm^u_p (\R)$ is the collection of all functions $f \in \cm_p^u (\R) $
such that all distributional derivatives  $D^\alpha f$ of order $|\alpha|\le m$ belong to $\cm_p^u (\R) $.
We put
\[
\|\, f\, |W^m \cm_p^u (\R) \| := \sum_{|\alpha|\le m} \|\, D^\alpha f\, |\cm_p^u(\R)\| \, .
\]
\end{defi}

\noindent
It will be convenient to use $W^0 \cm_p^u (\R):=  \cm_p^u (\R)$.

\begin{lem}\label{LP}
Let  $1 < p \le u <\infty$ and $m \in \N_0$.
Then $\ce^m_{u,p,2}(\R) = W^m \cm_p^u (\R)$ in the sense of equivalent norms.
\end{lem}

\begin{proof}
Mazzucato has proved the Littlewood-Paley characterization of  Morrey spaces in \cite{ma01}, i.e., she proved that
$\ce^0_{u,p,2} (\R) = \cm_p^u (\R)$, $1< p\le u < \infty$, holds in the sense of equivalent norms.
Combined with Lemma \ref{lift} this proves Lemma \ref{LP}.
\end{proof}

\begin{rem}
 \rm
 The spaces $\ce^s_{u,p,2}(\R)$ with $1<p< u< \infty$ and $s\in \re$ are investigated
 in Adams \cite{Adams}, see also Adams, Xiao \cite{AX} and Triebel \cite[Rem.~3.68]{t14}.
\end{rem}

For the next result we refer to \cite[Prop.~2.6]{ysy}.

\begin{lem}\label{emb}
Let $s\in\re$,  $1 \le   p \le u < \infty$ and $1 \le q \le \infty$.
Then
\be\label{limit0}
\ce^s_{u,p,q}(\R) \hookrightarrow   {B}^{s-d/u}_{\infty,\infty}(\R) \, .
\ee
\end{lem}

\begin{rem}\label{embbb}
 \rm
 {\rm (i)}
 It is well-known that ${B}^{s-d/u}_{\infty,\infty}(\R)\hookrightarrow L_\infty (\R)$ holds if  $s>d/u$.
 \\
 {\rm (ii)} Also in case of the Sobolev-Morrey spaces
 $ W^m \cm_p^u (\R)$ one knows that
 ${W}^{m} \cm_1^u (\R)\hookrightarrow L_\infty (\R)$ if $m >d/u$, see
 \cite{dch}.
\end{rem}

\subsection*{An important inequality}

Later on we shall need the following lemma, see \cite[Thm.~2.4]{sat1}.
For  $ \nu \in \mathbb{R} $ let  $ H_{2}^{\nu}  (\R) $ denote  the Bessel-potential space, defined as the collection of all $f\in \cs'(\R)$  with
\[
 \Vert f \vert H_{2}^{\nu}  (\R) \Vert =  \Vert (1 + \vert \, \cdot \,  \vert ^{2} )^{\frac{\nu}{2}} ( \cf f )(\, \cdot \, )  \vert L_{2}(\R)  \Vert <\infty \, .
\]

\begin{lem}\label{lsat}
Let $ 1\le  q \leq \infty $, $ 1 \le  p \leq u  < \infty $ and
$ \nu >  \frac{3d}{2} $.  Let $(R_j)_{j=0}^\infty \subset [1,\infty)$.
Suppose $ (h_j )_{j=0}^\infty \subset H_{2}^{\nu}  (\R)$ and  $ (f_j)_{j=0}^\infty \subset \mathcal{M}^{u}_{p}( \mathbb{R}^d)$  such that
$ \supp \cf f_j \subset B(0,R_j)$.
Then there is a constant $ c > 0 $, independent of $(R_j)_{j=0}^\infty$,
$(h_j)_{j=0}^\infty $ and  $(f_j)_{j=0}^\infty$, such that
\begin{align*}
& \Big\|\Big(\sum_{j=0}^\infty |\cfi [h_j \cf f_j](\, \cdot \, )|^q \Big)^{\frac{1}{q}}\Big| \cm^u_p (\R)\Big\| \\
& \qquad \qquad \le   c \,  \Big(\sup_{j \in \mathbb{N}_{0} }\, \Vert h_j(\, R_j \cdot\, ) \vert H_{2}^{\nu}  (\R) \Vert\Big)  \, \Big\|\Big(\sum_{j=0}^\infty |\, f_j (\, \cdot \, )\, |^q \Big)^{\frac{1}{q}}\Big| \cm^u_p (\R)\Big\|
\end{align*}
holds.
\end{lem}

\begin{rem}
 \rm Those vector-valued Fourier multiplier assertions are standard tools in the theory of function spaces, see,
e.g, \cite[1.6.3]{t83} for the classical case $p=u$.
\end{rem}


\subsection{Spaces on domains}
\label{domain}

In our article spaces on domains are defined by restrictions.
For us this is the most convenient way.
Here, for all domains $\Omega\subset \rn$ and  $g\in \cs'(\rn)$ by
$g_{|_\Omega}$ we denote the restriction of $g$ to $\Omega$.

\begin{defi}\label{d5.12}
Let $X (\rn)$ be a normed space of tempered distributions such that $X(\rn) \hookrightarrow \cs'(\rn)$.
Let $\Omega $ denote an open, nontrivial subset of $\rn$.
Then  $X(\Omega)$ is defined as the collection of all $f \in \cd' (\Omega)$ such that
there exists a distribution $g \in X(\rn)$ satisfying
\[
f (\varphi) = g (\varphi) \qquad \mbox{for all} \quad \varphi \in \cd (\Omega) \, .
\]
Here $\varphi \in \cd (\Omega)$ is extended by zero on $\rn\setminus \Omega$.
We put
\[
\| \, f\, \vert X(\Omega) \| := \inf \Big\{\| \, g\, \vert X(\rn) \|: \quad g_{|_\Omega} =f  \Big\} \, .
\]
\end{defi}

\noindent
Clearly, in the case of Morrey spaces this means the following.
\\
Let $1\le  p \le u < \infty$ and $\Omega \subset \rn$ be bounded.
Then the \emph{Morrey space $\cm^{u}_p (\Omega)$} is the collection  of all
 $f \in L_p^{\ell oc} (\Omega)$ such that
\[
\| \, f \, \vert \cm^{u}_p (\Omega) \| :=
\sup_{x \in \Omega} \sup_{r\in(0,\fz)} \, |B(x,r) \cap \Omega|^{ \frac{1}{u} -  \frac{1}{p}} \lf[
\int_{B(x,r)\cap \Omega} |f(y)|^p\, dy\r]^{\frac{1}{p}} <\infty\, .
\]

In this paper we will concentrate on Lipschitz domains $\Omega\subset \rn$. We follow Stein, see \cite[VI.3.2]{Stein}.

\begin{defi}\label{lipdo}
By a Lipschitz domain, we mean either a special or a bounded Lipschitz
domain.
\\
{\rm (i)}
A \emph{special Lipschitz domain} is an open set $\Omega\subset
\rn$ lying above the graph of a Lipschitz function $\omega:\ \rr^{d-1}\to\rr$, namely,
$$\Omega:=\{(x',x_d)\in\rn:\ x_d>\omega (x')\},$$
where $\omega$ satisfies that, for all $x',\ y'\in \rr^{d-1}$,
$$|\omega(x')-\omega(y')|\le A |x'-y'|$$
with a positive constant $A$ independent of $x'$ and $y'$.
\\
{\rm (ii)}
A \emph{bounded  Lipschitz domain} is a bounded domain $\Omega\subset \rn$
whose boundary $\partial \Omega$ can be covered by a finite number of open balls $B_k$
such that, for each $k$, after a suitable rotation, $\partial\Omega\cap B_k$
is a part of the graph of a Lipschitz function.
\end{defi}

For notational simplicity we shall use the convention, that a bounded Lipschitz domain
in $\re$ is just a bounded interval.


\section{The diamond space associated to ${\ce}^{s}_{u,p,q} (\rn)$
\label{s-diamon}}


In this section we will investigate the properties of the spaces $\accentset{\diamond}{\ce}^{s}_{u,p,q} (\rn)$, see
Definition \ref{d-cm}. This is very important in order to prove our main results. First, we recall two results from \cite{ysy3}, see Lemmas 2.25. and 2.26.

\begin{lem}\label{diamond2}
Let  $s\in\rr$, $1\le p \le u < \infty$ and $1\le q \le \infty$. Then $\mathring{\ce}_{u,p,q}^{s}(\rn) = \accentset{\diamond}{\ce}^{s}_{u,p,q} (\rn)$
if and only if $u=p$.
\end{lem}

Even more important is the following.

\begin{lem}\label{diamond3}
Let  $s\in\rr$, $1\le p \le u < \infty$ and $1\le q \le \infty$.
Then
\[
\accentset{\diamond}{\ce}^{s}_{u,p,q} (\rn) =  {\ce}^{s}_{u,p,q} (\rn)\qquad \mbox{if and only if}\qquad  u=p
\ \ \mbox{and}\ \  q\in [1,\infty).
\]
\end{lem}

\begin{rem}\label{gleich}
 \rm
 In particular this implies
 $\accentset{\diamond}{F}^{s}_{p,q} (\rn) =  \mathring{F}_{p,q}^{s}(\rn) =
 {F}^{s}_{p,q} (\rn)$ if $1\le p,q<\infty$.
\end{rem}

Now we turn to some further descriptions of the diamond spaces.
The diamond spaces are defined as a closure. So it is most natural to look for characterizations in form of approximations.
A first characterization is using the Littlewood-Paley decomposition.


\subsection{A characterization using the Littlewood-Paley decomposition}


Let $(\varphi_j)_{j=0}^\infty$ be a smooth dyadic decomposition of unity.
Then we put
\[
 S^N f (x):= \sum_{j=0}^N \cfi [\varphi_j \cf f ](x)\, , \qquad N \in \mathbb{N}_{0} .
\]
Of course, by the Paley-Wiener-Schwarz Theorem,
$S^N f$ are smooth functions.

\begin{lem}\label{properties}
Let $1\le p\le u  < \infty$, $1\le q \le \infty$ and $s\in \re$.
Let $f \in {\ce}^{s}_{u,p,q} (\rn)$. Then the sequence
$(S^N f)_{N=0}^\infty$ has the following properties:
\begin{itemize}
 \item[(i)] $S^Nf \in {\ce}^{\sigma}_{u,p,q} (\rn)$ for all $\sigma \in \re$.
 \item[(ii)] For all $\alpha \in \N_0^d$ we have $D^\alpha (S^N f) \in {\ce}^{s}_{u,p,q} (\rn)$.
 \item[(iii)] For all $\alpha \in \N_0^d$ we have $D^\alpha (S^N f) \in L_\infty (\R)$.
 \item[(iv)] The following identity holds
 \[
  S^N f (x) = \cfi [\varphi_0 (2^{-N}\xi)\cf f(\xi)](x)\, , \qquad x \in \R, \quad N \in \mathbb{N}_{0} .
 \]
  \item[(v)] There exists a constant $c$, independent on $f$,  such that
\be\label{ws-11}
\sup_{N \in \mathbb{N}_{0}}\,
\| \, S^N f \, |{\ce}^{s}_{u,p,q} (\rn)\| \le c \, \| \,  f \, |{\ce}^{s}_{u,p,q} (\rn)\|\, .
\ee
\end{itemize}
\end{lem}

\begin{proof}
Part (i) is a consequence of the estimate
\[
 \| \, S^N f  \, |{\ce}^{\sigma}_{u,p,q} (\rn)\|
\le  c\,
\Big\| \Big(\sum_{j=0}^{N+1} 2^{j\sigma q} |\cfi [\varphi_j \, \cf f](\, \cdot\, )|^q\Big)^{\frac{1}{q}}\Big|
\cm^u_p (\R)\Big\|
\]
with some $c$ independent of $f$ and $N \in \mathbb{N}_{0} $, see Lemma \ref{lsat}.
From (i) we derive that
$S^N f \in {\ce}^{s+m}_{u,p,q} (\rn)$ with $ m \in \mathbb{N}_{0}  $. Next we use Lemma \ref{lift} obtaining
$D^\alpha (S^N f) \in {\ce}^{s}_{u,p,q} (\rn)$ for $|\alpha|=m$.
To show (iii) it is enough to apply Lemma \ref{emb}.
The next part (iv) is an elementary conclusion of the definition of the functions $\varphi_j$ with
$j \in \{ 1,2,\ldots  , N \}$.
Finally, (v) follows from the generalized Minkowski inequality combined with a standard convolution inequality:
\beqq
\| \, S^N f \, |{\ce}^{s}_{u,p,q} (\rn)\|
&  =  & \| \, \cfi [\varphi_{0} (2^{-N}\xi)\, \cf f(\xi)](\, \cdot\, ) \, |{\ce}^{s}_{u,p,q} (\rn)\|
\\
&\le &
 \| \, \cfi \varphi_{0}\, |L_1 (\R)\|\,
 \|\,  f \, |{\ce}^{s}_{u,p,q} (\rn)\|\, .
\eeqq
The proof is complete.
\end{proof}

Associated to the definition of $\accentset{\diamond}{\ce}^{s}_{u,p,q} (\rn)$ we need a further abbreviation.

\begin{defi}\label{diamond7}
 Let $1\le p < u <\infty$, $1 \le  q \le \infty$ and $s \geq 0$.
The set $E^s_{u,p,q} (\R)$ is the collection of all functions $f \in \ce^s_{u,p,q} (\R)$
such that $D^\alpha f \in \ce^s_{u,p,q} (\R)$ for all $\alpha \in \N_0^d$.
\end{defi}

As an immediate consequence we get
\[
 \overline{E^s_{u,p,q} (\R)}^{\|\, \cdot \, |\ce^s_{u,p,q} (\R)\|} = \accentset{\diamond}{\ce}^{s}_{u,p,q} (\rn) \, .
\]
Moreover, by  Lemma \ref{properties}, for any $f\in \ce^s_{u,p,q} (\R)$ we have $S^N f \in E^s_{u,p,q} (\R)$.
This will be of some use later on.

\begin{prop}\label{diamond1}
Let $1\le p \le u <  \infty$, $1\le q \le \infty$ and $s\in \re$.
Then $\accentset{\diamond}{\ce}^{s}_{u,p,q} (\rn)$
 is the collection of all $f\in {\ce}^{s}_{u,p,q} (\rn)$
 such that
 \be\label{ws-10}
  \lim_{N\to \infty} \, \| \, f - S^N f\, |{\ce}^{s}_{u,p,q} (\rn)\| =0\, .
 \ee
\end{prop}

\begin{proof}
Clearly, if \eqref{ws-10} holds, then   $f\in \accentset{\diamond}{\ce}^{s}_{u,p,q} (\rn)$
follows.\\
Now, let us suppose that  $f\in \accentset{\diamond}{\ce}^{s}_{u,p,q} (\rn)$.
By $(f_\ell)_\ell$ we denote a sequence in $E^s_{u,p,q}(\R)$ such that
\[
\lim_{\ell \to \infty} \, \| \, f -  f_\ell\, |{\ce}^{s}_{u,p,q} (\rn)\| =0 \, .
\]
Without loss of generality we may assume
\[
 \| \, f -  f_\ell\, |{\ce}^{s}_{u,p,q} (\rn)\| < \frac 1 \ell \, , \qquad \ell \in \N\, .
\]
Let $\sigma \in \re$ with $\sigma > s$. We use a standard Fourier multiplier assertion from Lemma \ref{lsat}. Then we obtain
\beqq
 \| \, f_\ell &-&  S^N f_\ell\, | {\ce}^{s}_{u,p,q} (\rn)\| =  \Big\|\, \sum_{j=N+1}^\infty
 \cfi [\varphi_j \cf f_\ell ](\, \cdot\, ) \Big|{\ce}^{s}_{u,p,q} (\rn) \Big\|
\\
&\le & c_1\,
\Big\|\, \Big(\sum_{j=N}^\infty 2^{jsq} |\cfi [\varphi_j \cf f_\ell ](\, \cdot\, )|^q
\Big)^{\frac{1}{q}} \Big|{\cm}^{u}_{p} (\rn) \Big\|
\\
&\le & c_2\, 2^{N(s-\sigma)}\,
\Big\|\, \Big(\sum_{j=N}^\infty 2^{j\sigma q} |\cfi [\varphi_j \cf f_\ell ](\, \cdot\, )|^q\Big)^{\frac{1}{q}} \Big|{\cm}^{u}_{p} (\rn) \Big\|
\\
& \le & c_2\, 2^{N(s-\sigma)}\,  \|\, f_\ell\, | {\ce}^{\sigma}_{u,p,q} (\rn)\|\, .
\eeqq
This shows that
\[
 \lim_{N \to \infty} \, \| \, f_\ell  -  S^N f_\ell\, | {\ce}^{s}_{u,p,q} (\rn)\| =0 \qquad \mbox{for any}
 \quad \ell\in \N.
\]
Hence, for $\ell \in \N$ there exists some $N(\ell)$ such that
\[
 \| \, f_\ell - S^{N(\ell)} f_\ell\, |{\ce}^{s}_{u,p,q} (\rn)\| < \frac 1 \ell \, .
\]
This yields
\beqq
\| \, f  -  S^{N(\ell)} f\, |{\ce}^{s}_{u,p,q} (\rn)\|  &\le &
\| \, f - f_\ell\, |{\ce}^{s}_{u,p,q} (\rn)\| +
\| \, f_\ell - S^{N(\ell)} f_\ell\, |{\ce}^{s}_{u,p,q} (\rn)\|
\\
&& \quad   + \quad \| \, S^{N(\ell)}f_\ell  - S^{N(\ell)} f\, |{\ce}^{s}_{u,p,q} (\rn)\|
\\
&\le &
\frac {2}{\ell} +
\| \, S^{N(\ell)} (f_\ell  -f)\, |{\ce}^{s}_{u,p,q} (\rn)\|
\\
&\le &
\frac {2+c}{\ell} \, ,
\eeqq
where $c$ is the constant from \eqref{ws-11}.
Hence, we have the convergence of an appropriate subsequence $(S^{N(\ell)}f)_{\ell=1}^\infty$.
It remains to switch from a  subsequence to the whole sequence.
Therefore we assume that our sequence $(N(\ell))_\ell$ satisfies
$$N(\ell+1)-N(\ell) > 5\quad {\rm for\ all}\ \ell.$$
Furthermore we will use the following identity
\beqq
\Big\| \sum_{j=M}^N \cfi [\varphi_j \, \cf f]\, && \hspace{-0.2cm} \Big |{\ce}^{s}_{u,p,q} (\rn) \Big \|
=  \bigg\| \bigg(\sum_{m=M+1}^{N-1} 2^{msq} |\cfi [\varphi_m \, \cf f](\, \cdot\, )|^q
\\
& + &
2^{Msq} |\cfi [\varphi_{M} \, (\varphi_M + \varphi_{M+1})\,  \cf f](\, \cdot\, )|^q
\\
& + & 2^{(M-1)sq} |\cfi [\varphi_{M-1}  \, \varphi_{M} \, \cf f](\, \cdot\, )|^q
\\
& + &
2^{Nsq} |\cfi [\varphi_N \, (\varphi_{N-1} + \varphi_N)\,  \cf f](\, \cdot\, )|^q
\\
& + & 2^{(N+1)sq} |\cfi [\varphi_{N+1}\,  \varphi_N \, \cf f](\, \cdot\, )|^q
\bigg)^{\frac{1}{q}}\, \bigg|{\cm}^{u}_{p} (\rn)\bigg\|\, ,
\eeqq
valid for all natural numbers $M$ and $N$ such that $2\le M+1<  N-1$.
This follows from
\[
\varphi_m \, \cdot \Big( \sum_{j=M}^N \varphi_j \Big) = \left\{\begin{array}{lll}
\varphi_m & \qquad & \mbox{if}\quad M < m < N;\\
\varphi_{M-1}  \, \varphi_{M} && \mbox{if}\quad m=M-1;
\\
\varphi_{M}  \, (\varphi_{M} + \varphi_{M+1}) && \mbox{if}\quad m=M;
\\
\varphi_{N}  \, (\varphi_{N-1} + \varphi_N) && \mbox{if}\quad m=N;
\\
\varphi_{N+1}  \, \varphi_{N} && \mbox{if}\quad  m=N + 1;
\\
0&& \mbox{otherwise}\, .
\end{array}
\right.
\]
A standard convolution inequality  combined with the generalized Minkowski inequality yields
\begin{align*}
&\|\, \cfi [\varphi_{j}\,  \varphi_\ell \, \cf f](\, \cdot\, )\, |{\cm}^{u}_{p} (\rn)\|\\
&\quad \le
 \| \, \cfi \varphi_{j}\, |L_1 (\R)\|\,
 \|\, \cfi [ \varphi_\ell \, \cf f](\, \cdot\, )\, |{\cm}^{u}_{p} (\rn)\|\, .
\end{align*}
Applying a  homogeneity argument we find
\[
\| \, \cfi \varphi_{j}\, |L_1 (\R)\| =
\| \, \cfi \varphi_{1}\, |L_1 (\R)\| , \qquad j \in \mathbb{N}  .
\]
Alltogether this shows
\begin{align}\label{ws-20}
&\Big\| \Big(\sum_{m=M+1}^{N-1} 2^{msq}
|\cfi [\varphi_m \, \cf f](\, \cdot\, )|^q\Big)^{\frac{1}{q}}\Big|\cm^u_p (\R)\Big\|
\nonumber
\\
&\quad \le
\Big\| \sum_{j=M}^N \cfi [\varphi_j \, \cf f]\, \Big |{\ce}^{s}_{u,p,q} (\rn) \Big  \|
\nonumber
\\
&\quad \le c_3\,
\Big\| \Big(\sum_{m=M-1}^{N+1} 2^{msq} |\cfi [\varphi_m \, \cf f](\, \cdot\, )|^q\Big)^{\frac{1}{q}}\Big|
\cm^u_p (\R)\Big\|
\end{align}
with some constant $c_3$ independent on $M,N$ and $f$.
Let
\[
1 \le N(\ell) \le M-3 < N-3 < N+2 \le N(\ell+ 1)\, .
\]
Then \eqref{ws-20} implies
\beqq
\| \, S^N f &-& S^{M-1} f \, |{\ce}^{s}_{u,p,q} (\rn)\|  =
\Big\| \sum_{j=M}^N  \cfi [\varphi_j \, \cf f]\, \Big|{\ce}^{s}_{u,p,q} (\rn)\Big\|
\\
&\le & c_3\,
\Big\| \Big(\sum_{m=M-1}^{N+1} 2^{msq} |\cfi [\varphi_m \, \cf f]\, |^q\Big)^{\frac{1}{q}}\Big|
\cm^u_p (\R)\Big\|
\\
&\le & c_3\,
\Big\| \sum_{j=N(\ell)+1}^{N(\ell+1)}  \cfi [\varphi_j \, \cf f]\, \Big|{\ce}^{s}_{u,p,q} (\rn)\Big\|
\\
&=& c_3\,
\| \, S^{N(\ell+1)} f - S^{N(\ell)} f \, |{\ce}^{s}_{u,p,q} (\rn)\|\, .
\eeqq
Repeating the argument we conclude that
\[
 \| \, S^N f - S^{M-1} f \, |{\ce}^{s}_{u,p,q} (\rn)\| \le c_{4}\,
\| \, S^{N(\ell+2)} f - S^{N(\ell-1)} f \, |{\ce}^{s}_{u,p,q} (\rn)\|\,
\]
for all $M,N$ such that $N(\ell) \le M \le N \le N(\ell+1)$ with $c_4$ independent of $\ell$.
Consequently $(S^N f)_{N=0}^\infty$ is a Cauchy sequence in ${\ce}^{s}_{u,p,q} (\rn)$.
This proves the claim.
\end{proof}

\begin{rem}
 \rm
 Proposition \ref{diamond1} is not new, we refer to
Hakim, Nogayama  and  Sawano \cite[Thm.~1.1]{hns19}.
However, our proof is slightly different and covers the cases $p=1\le u < \infty$.
In fact, it extends without any change to $0 <p\le u < \infty$.
\end{rem}

Later on we shall need the following consequence of Proposition \ref{diamond1}.

\begin{prop}\label{diamond9}
Let $1\le p \le u <  \infty$, $1\le q_0,q_1 \le \infty$ and $s_0,s_1\in \re$ with $s_1 < s_0$.
Then we have the continuous embedding
\[
{\ce}^{s_0}_{u,p,q_0} (\rn)
\hookrightarrow
\accentset{\diamond}{\ce}^{s_1}_{u,p,q_1} (\rn)\, .
 \]
\end{prop}

\begin{proof}
Let $f \in {\ce}^{s_0}_{u,p,q_0} (\rn)$.
 Lemma  \ref{lsat} yields
 \beqq
\| f-S^N f\, |{\ce}^{s_1}_{u,p,q_1} (\rn)\| & \le & c_1
\, \Big\|\Big(\sum_{j=N}^\infty 2^{js_1q_1} |\cfi [\varphi_j \cf f ]\, |^{q_1}\Big)^{\frac{1}{q_1}}\, \Big|\cm_p^u (\R)\Big\|
\\
&\le & c_2 \, 2^{-(s_0-s_1)N} \,
\, \Big\|\, \sup_{j\ge N}\, 2^{js_0} |\cfi [\varphi_j \cf f ]\, |\, \Big|\cm_p^u (\R)\Big\|
 \eeqq
 with constants $c_1,c_2$ independent of $f$ and $N$.
 Because of ${\ce}^{s_0}_{u,p,q_0} (\rn) \hookrightarrow {\ce}^{s_0}_{u,p,\infty} (\rn)$,
 for $N \to \infty$ this implies \eqref{ws-10} and therefore $f\in \accentset{\diamond}{\ce}^{s_1}_{u,p,q_1} (\rn)$.
 \end{proof}


\subsection{A characterization using mollifiers}


It is possible to describe the spaces $ \accentset{\diamond}{\ce}^{s}_{u,p,q} (\rn)  $ by using mollifiers. In this section we will briefly collect the main ideas concerning that topic. For that purpose we need some more notation. Therefore let $\varrho \in \cd (\R)$ be a function satisfying
\[
 \int_{\R} \varrho (x)\, dx =1 \qquad \mbox{and}\qquad \supp \varrho \subset B(0,1)\, .
\]
We put $\rho_j (x) := 2^{jd} \varrho (2^jx)$ with $x \in \R$ and $j \in \N$.
For a Banach space  $X$ that is continuously embedded into $ \cs' (\R)$ we define $X^{\ell oc}$ as the collection of all $f\in \cs' (\R)$ such that the pointwise product fulfills $\psi \, \cdot \, f \in X$ for all $\psi \in \cd (\R)$.
Convergence of a sequence $\{f_j\}_{j=1}^\infty$ with limit $f$ in $X^{\ell oc}$ is defined as
\[
\lim_{j\to \infty}\, \|\, f\, \psi - f_j\, \psi|X\|=0  \qquad \mbox{for all}\quad \psi \in \cd (\R)\,.
\]

\begin{lem}\label{properties2}
Let $1\le p \leq u < \infty$, $1\le q \le \infty$ and $s>0$.
Let $f \in {\ce}^{s}_{u,p,q} (\rn)$. Then the sequence
$\{f \ast \rho_j\}_{j=1}^\infty$ has the following properties:
\begin{itemize}
  \item[(i)] For all $\alpha \in \N_0^d$ and all $j \in \N$ we have $D^\alpha (f\ast \rho_j) \in {\ce}^{s}_{u,p,q} (\rn)$, i.e.,
  $(f\ast \rho_j) \in E^s_{u,p,q}(\R)$.
 \item[(ii)] For all $\alpha \in \N_0^d$  and all $j \in \N$ we have $D^\alpha (f\ast \rho_j) \in L_\infty (\R)$.
 \item[(iii)]  For all $j \in \N$ we have $(f\ast \rho_j) \in C^\infty (\R)$.
 \item[(iv)]  For all $j \in \N$ we have $(f\ast \rho_j) \in {\ce}^{\sigma}_{u,p,q} (\rn)$ for all $\sigma \in \re$.
 \item[(v)] There exists a constant $c$, independent on $f$,  such that
\be\label{ws-11b}
\sup_{j \in \mathbb{N}}\,
\| \, f\ast \rho_j \, |{\ce}^{s}_{u,p,q} (\rn)\| \le c \, \| \,  f \, |{\ce}^{s}_{u,p,q} (\rn)\|\, .
\ee
\end{itemize}
\end{lem}

\noindent
Essentially all of Lemma \ref{properties2} is known. So we skip the proof. There is a counterpart of Proposition \ref{diamond1} that
reads as follows.

\begin{prop}\label{diamond4}
 Let $1\le p \le u < \infty$, $1\le q \le \infty$ and $s > 0$. Let $f \in {\ce}^{s}_{u,p,q} (\rn)$.
 Then the following assertions are equivalent.
 \begin{itemize}
  \item[(i)] $f \in \accentset{\diamond}{\ce}^{s}_{u,p,q} (\rn)$;
  \item[(ii)] $\lim_{j \to \infty} \, \| \, f\ast\varrho_j-f\, |\ce^{s}_{u,p,q}(\R)\| =0$.
 \end{itemize}
\end{prop}

We will not use Proposition \ref{diamond4} in what follows.
Therefore we will drop the proof.


\subsection{A characterization in terms of differences}


It is possible to describe the spaces ${\ce}^{s}_{u,p,q} (\R)$ and also the spaces $  \accentset{\diamond}{\ce}^{s}_{u,p,q} (\R)  $ in terms of differences. For that purpose we need some additional notation. Let $ f : \mathbb{R}^d \rightarrow \mathbb{C}$ be a function. Then for $ x, h \in \mathbb{R}^d $ we define the difference of the first order by $ \Delta_{h}^{1}f (x) := f ( x + h ) - f (x) $. Let $ N \in \mathbb{N}$ with $ N > 1 $. Then we define the difference of the order $ N $ by $\Delta_{h}^{N}f (x) :=   (\Delta_{h} ^1  ( \Delta_{h} ^{N-1}f   ) ) (x) $. Now at first we recall the following characterization of
${\ce}^{s}_{u,p,q} (\rn)$.
We refer to \cite{Ho19} and \cite[4.3.1]{ysy}.

\begin{prop}\label{diff}
Let $1\le p < u <\infty$, $1 \le  q \le \infty$ and $s>0$. Let $N \in \N$ such that $s<N$. Then
${\ce}^{s}_{u,p,q} (\R)$ is the collection of all $f \in {\cm}^{u}_{p} (\R)$
such that
\beqq
\|\, f \, |\cm_{p}^u (\R)\| + \bigg\|\bigg(
  \int_0^1 t^{-sq} \Big( t^{-d} \int_{B(0,t)} |\Delta_h^N f(x)| \, dh\Big) ^q \frac{dt}{t}\bigg)^{\frac{1}{q}}\bigg| \cm^u_p (\R)\bigg\|
\eeqq
is finite (with equivalent norms). In the case $ q = \infty  $ the usual modifications have to be made.
\end{prop}

Next we turn to the spaces $ \accentset{\diamond}{\ce}^{s}_{u,p,q} (\R)  $. Very much in the spirit of Lemma \ref{morrey43} is the following observation.

\begin{lem}\label{lem1}
 Let $1\le p < u <\infty$, $1\le  q \le \infty$ and $s>0$. Let $N \in \N$ such that $s<N$. Then
$ \accentset{\diamond}{\ce}^{s}_{u,p,q} (\R)$ is contained  in the set of all $f \in {\ce}^{s}_{u,p,q} (\R)$
such that
\be\label{ws-21}
 \lim_{r \downarrow 0}
|B(y,r)|^{\frac{1}{u}- \frac{1}{p}}\Big(\int_{B(y,r)} |f(x)|^p\,dx\Big)^{\frac{1}{p}}   =  0
\ee
and
\beq\label{ws-22}
\lim_{r\downarrow  0} \, |B(y,r)|^{\frac 1u - \frac 1p} \, \bigg[\int_{B(y,r)} \bigg(
  \int_0^1 t^{-sq} \Big( t^{-d} \int_{B(0,t)} |\Delta_h^N f(x)| \, dh\Big) ^q \frac{dt}{t}\bigg)^{\frac{p}{q}}
 dx\bigg]^{\frac{1}{p}} =  0\, ,
\eeq
both uniformly in $y\in \R$.
\end{lem}

\begin{proof}
{\em Step 1.}
In a first step we deal with functions $f$ belonging to $E^s_{u,p,q} (\R)$. Clearly, those functions
are uniformly Lipschitz continuous on $\R$, see the proof of Lemma \ref{properties}.
To see  \eqref{ws-21} in this situation we argue as follows.
Obviously we have
\[
\Big(\int_{B(y,r)} |f(x)|^p\,dx\Big)^{\frac{1}{p}} \le   \|\, f\, |L_\infty (\R)\|\, |B(y,r)|^{\frac{1}{p}} .
\]
Multiplying this inequality by $|B(y,r)|^{1/u-1/p}$ it follows for $u< \infty$ that
the right-hand side tends to $0$ (uniformly in $y$) if $r\downarrow 0$.
The argument for deriving \eqref{ws-22} is quite similar.
Recall that for a smooth function we have with $ N \in \mathbb{N}   $
\[
|\Delta_h^N f(x)| \le c_1\, \Big(\max_{|\alpha|\le N} \sup_{y\in \R} \, | \, D^\alpha f (y)\, |\Big) \, |h|^N\, , \qquad x,h \in \R\, ,
\]
with a constant $c_1$ independent of $f,x$ and $h$.
Hence
\[
\bigg(
  \int_0^1 t^{-sq} \Big( t^{-d} \int_{B(0,t)} |\Delta_h^N f(x)| \, dh\Big) ^q \frac{dt}{t}\bigg)^{\frac{1}{q}} \le  c_2\, \Big(
  \int_0^1 t^{-sq} t^{Nq} \, \frac{dt}{t}\Big)^{\frac{1}{q}}\le c_3 < \infty
\]
for some $c_3$ independent of $x$.
This implies
\beqq
|B(y,r)|^{\frac 1u - \frac 1p}\,
\bigg[\int_{B(y,r)}  \bigg(
  \int_0^1 t^{-sq} \Big( t^{-d} \int_{B(0,t)} |\Delta_h^N f(x)| \, dh\Big) ^q \frac{dt}{t}\bigg)^{\frac{p}{q}}
 dx\bigg]^{\frac{1}{p}} \le  c_3 \, |B(y,r)|^{\frac 1u}
\eeqq
and therefore the claim follows.
\\
{\em Step 2.} Now we turn to the general case.
Let $f \in \accentset{\diamond}{\ce}^s_{u,p,q} (\R)$ and let $\varepsilon >0$  be given.
Then with $ M \in \mathbb{N}  $ it follows
\begin{align*}
& |B(y,r)|^{\frac 1u - \frac 1p}\,
\bigg[\int_{B(y,r)}  \bigg(
  \int_0^1 t^{-sq} \Big( t^{-d} \int_{B(0,t)} |\Delta_h^N f(x)| \, dh\Big) ^q \frac{dt}{t}\bigg)^{\frac{p}{q}}
 dx\bigg]^{\frac{1}{p}}
 \\
& \qquad \le  \|\, f-S^M f\, |{\ce}^s_{u,p,q} (\R)\| \\
& \qquad \qquad + |B(y,r)|^{\frac 1u - \frac 1p}\bigg[\int_{B(y,r)}  \bigg(
  \int_0^1 t^{-sq} \Big( t^{-d} \int_{B(0,t)} |\Delta_h^N (S^M f)(x)| \, dh\Big) ^q \frac{dt}{t}\bigg)^{\frac{p}{q}}
 dx\bigg]^{\frac{1}{p}} \, .
 \end{align*}
 The second term on the righ-hand side becomes smaller than $\varepsilon >0$ if $r\le r_0 (\varepsilon)$ since $S^M f \in E^s_{u,p,q}(\R)$ and therefore we may use Step 1.
 The first term on the right-hand side will be smaller than $\varepsilon >0$ if $M\ge  M_0 (\varepsilon)$ thanks to Proposition \ref{diamond1}.
 Both statements hold uniformly in $y$.
 This proves \eqref{ws-22}.
 The convergence in \eqref{ws-21} can be proved in a similar way.
\end{proof}

Now we turn to the converse.

\begin{lem}\label{lem2}
 Let $1\le p < u <\infty$, $1\le  q < \infty$ and $s>0$. Let $N \in \N$ such that $s<N$.
 Let $f \in {\ce}^{s}_{u,p,q} (\R)$ be a function with compact support and such that
\eqref{ws-21}, \eqref{ws-22} hold uniformly in $y\in \R$.
Then $f\in  \accentset{\diamond}{\ce}^{s}_{u,p,q} (\R)$.
\end{lem}

\begin{proof}
 Because of Proposition \ref{diamond1} it is enough to prove
\[
\lim_{M \to \infty} \Vert f - S^M f \vert {\ce}^{s}_{u,p,q} (\R) \Vert =0 .
\]
Using Proposition \ref{diff} this can be reduced  to show
\begin{equation}\label{di_dif_2_e2}
\lim_{M \to \infty} \Vert f - S^M f \vert \mathcal{M}^{u}_{p}(\R) \Vert =0
\end{equation}
and
\begin{equation}\label{di_dif_2_e3}
\lim_{M \to \infty} \Big \Vert \Big (
  \int_0^1 t^{-sq} \Big( t^{-d} \int_{B(0,t)} |\Delta_h^N (f - S^M f)(x)| \, dh\Big) ^q \frac{dt}{t}\Big)^{\frac{1}{q}}\Big| \mathcal{M}^u_p (\R)\Big \Vert = 0.
\end{equation}
{\em Step 1.} We shall show \eqref{di_dif_2_e2}. Let $ f \in  {\ce}^{s}_{u,p,q} (\R) $. Let $ M \in \mathbb{N}  $ and $ 0 < \sigma < s   $. Then we find
\beqq
\| \, f-S^M f\, |  \cm^u_p (\R) \| & \le &  \, \Big \| \sum_{j= M + 1}^\infty
|\cfi [\varphi_j \cf f](\, \cdot \, )|\, \Big |\cm^u_p (\R) \Big  \|
\\
& \leq &   \, c_{1} 2^{-M \sigma} \Big \| \sum_{j= M + 1}^\infty 2^{j \sigma}
|\cfi [\varphi_j \cf f](\, \cdot \, )|\, \Big |\cm^u_p (\R) \Big  \|
\\
& \le &   \, c_{1} 2^{-M \sigma } \,  \| \, f \, |\ce^{\sigma}_{u,p,1} (\R)\|  \, ,
\\
& \le &   \, c_{2} 2^{-M \sigma } \,  \| \, f \, |\ce^{s}_{u,p,q} (\R)\|  \, .
\eeqq
We used Definition \ref{LTMS}.
Here $c_2$ is independent of $f$ and $M \in \mathbb{N} $.
So because of $ f \in  {\ce}^{s}_{u,p,q} (\R)   $ if $ M $ tends to infinity  \eqref{di_dif_2_e2} follows.
\\
{\em Step 2.} Next we prove \eqref{di_dif_2_e3}. Let $B$ stand for every ball in $\R$.
Since $ f $ satisfies \eqref{ws-22}, for every $ \varepsilon > 0 $, we find some  $ \delta > 0 $ such that
\begin{align*}
\sup_{|B| < \delta}   |B|^{\frac 1u - \frac 1p}\,
\bigg[\int_{B} \bigg(
  \int_0^1 t^{-sq} \Big( t^{-d} \int_{B(0,t)} |\Delta_h^N f(x)| \, dh\Big) ^q \frac{dt}{t}\bigg)^{\frac{p}{q}}
 dx\bigg]^{\frac{1}{p}} \leq \varepsilon .
\end{align*}
The generalized Minkowski inequality and a standard convolution inequality yield
\beqq
\bigg[\int_{B} \bigg(&& \hspace{-0.4cm}
  \int_0^1 t^{-sq} \Big( t^{-d} \int_{B(0,t)} |\Delta_h^N S^M f(x)| \, dh\Big) ^q \frac{dt}{t}\bigg)^{\frac{p}{q}}
 dx\bigg]^{\frac{1}{p}}
 \\
 &\le & c_3 \, \bigg[\int_{B} \bigg(
  \int_0^1 t^{-sq} \Big( t^{-d} \int_{B(0,t)} |\Delta_h^N  f(x)| \, dh\Big) ^q \frac{dt}{t}\bigg)^{\frac{p}{q}}
 dx\bigg]^{\frac{1}{p}}
\eeqq
with $c_3$ independent of $B$ and $f$. Consequently we get
\beq\label{wss-00}
&&\hspace*{-0.5cm} \Big \Vert \Big (
  \int_0^1 t^{-sq} \Big( t^{-d} \int_{B(0,t)} |\Delta_h^N (f - S^M f)(x)| \, dh\Big) ^q \frac{dt}{t}\Big)^{\frac{1}{q}}\Big| \mathcal{M}^u_p (\R)\Big \Vert
\\
 & & \hspace*{-0.5cm} \leq c_4 \,  \varepsilon  +   \sup_{|B| \geq \delta}   |B|^{\frac 1u - \frac 1p}\,
\bigg[\int_{B} \bigg(
  \int_0^1 t^{-sq} \Big( t^{-d} \int_{B(0,t)} |\Delta_h^N (f - S^{M} f)(x)| \, dh\Big) ^q \frac{dt}{t}\bigg)^{\frac{p}{q}}
 dx\bigg]^{\frac{1}{p}}. \nonumber
\eeq
Since $f \in {\ce}^{s}_{u,p,q} (\R)$ the supremum on the right-hand side is finite.
By the definition of the supremum there exists a sequence of balls $B_j:=B(y_{j},r_{j})$ with $j \in \N$ and $ \vert B(y_{j},r_{j})   \vert \geq \delta   $ such that
\beq\label{wss-01}
&& \sup_{|B| \geq \delta}   |B|^{\frac 1u - \frac 1p}\,
\bigg[\int_{B} \bigg(
  \int_0^1 t^{-sq} \Big( t^{-d} \int_{B(0,t)} |\Delta_h^N (f - S^{M} f)(x)| \, dh\Big) ^q \frac{dt}{t}\bigg)^{\frac{p}{q}}
 dx\bigg]^{\frac{1}{p}} \nonumber
 \\
&& \qquad  \leq  \frac{1}{j} +   |B_j|^{\frac 1u - \frac 1p}\,
\bigg[\int_{B_j} \bigg(
  \int_0^1 t^{-sq} \Big( t^{-d} \int_{B(0,t)} |\Delta_h^N (f - S^{M} f)(x)| \, dh\Big) ^q \frac{dt}{t}\bigg)^{\frac{p}{q}}
 dx\bigg]^{\frac{1}{p}} \nonumber
 \\
&& \qquad \leq  \frac{1}{j} +   |B_j|^{\frac 1u - \frac 1p}\,
\bigg[\int_{\R} \bigg(
  \int_0^1 t^{-sq} \Big( t^{-d} \int_{B(0,t)} |\Delta_h^N (f - S^{M} f)(x)| \, dh\Big) ^q \frac{dt}{t}\bigg)^{\frac{p}{q}}dx\bigg]^{\frac{1}{p}} \nonumber
  \\
&& \qquad \leq  \frac{1}{j} +  c_5  \,   \delta ^{\frac 1u - \frac 1p}\,
\Vert  f - S^{M} f  \vert  F^{s}_{p,q}(\R)  \Vert.
\eeq
Here in the last step we used Proposition \ref{diff} for the original Lizorkin-Triebel spaces $ F^{s}_{p,q}(\R)  $, i.e., in case $p=u$, see also \cite[2.5.11]{t83}.
\\
{\em Substep 2.1}. We claim that a function $ f \in {\ce}^{s}_{u,p,q} (\R)  $ with
compact support belongs to  $    F^{s}_{p,q}(\R)  $ as well.
We may assume $\supp f \subset B(0,R)$ for some $R>1$. Based on Proposition \ref{diff}
we observe that
\beq\label{wss-02}
&&
\Vert f  \vert L_{p}(\R) \Vert  +  \Big \Vert \Big ( \int_0^1 t^{-sq} \Big( t^{-d} \int_{B(0,t)} |\Delta_h^N f(x)| \, dh\Big) ^q \frac{dt}{t}\Big)^{\frac{1}{q}}\Big| L_{p} (\R)\Big \Vert
\nonumber
\\
&& \qquad =   \Vert f  \vert L_{p}(B(0,R)) \Vert +  \Big \Vert \Big ( \int_0^1 t^{-sq} \Big( t^{-d} \int_{B(0,t)} |\Delta_h^N f(x)| \, dh\Big) ^q \frac{dt}{t}\Big)^{\frac{1}{q}}\Big| L_{p} (B(0,R + N) )\Big \Vert
\nonumber
\\
&& \qquad \leq    |B(0,R)|^{-\frac 1u + \frac 1p}  \Vert f  \vert \mathcal{M}^{u}_{p}(\R) \Vert
\nonumber
\\
&& \qquad \qquad +     |B(0,R+N)|^{-\frac 1u + \frac 1p} \Big \Vert \Big ( \int_0^1 t^{-sq} \Big( t^{-d} \int_{B(0,t)} |\Delta_h^N f(x)| \, dh\Big) ^q \frac{dt}{t}\Big)^{\frac{1}{q}}\Big| \mathcal{M}^{u}_{p}(\R) \Big \Vert
\nonumber
\\
&& \qquad \leq  c_6 \, (R + N)^{d(\frac 1p-\frac 1u )} \Vert f  \vert {\ce}^{s}_{u,p,q} (\R) \Vert.
\eeq
Hence $f \in   F^{s}_{p,q}(\R)$.
\\
{\em Substep 2.2.} Next we shall use Lemma \ref{diamond3} and Proposition \ref{diamond1}. Because of
$ f \in   F^{s}_{p,q}(\R) =\ce^s_{p,p,q}(\R) = \accentset{\diamond}{\ce}^{s}_{p,p,q} (\R)$
and $1\le  p, q < \infty  $ we get
\be\label{wss-03}
\lim_{M \rightarrow \infty } \Vert  f - S^{M} f  \vert  F^{s}_{p,q}(\R)  \Vert = 0.
\ee
Finally, we collect \eqref{wss-00}-\eqref{wss-03} together and find for fixed $\varepsilon$
and associated $\delta$
\beqq
&&\hspace*{-0.5cm} \Big \Vert \Big (
  \int_0^1 t^{-sq} \Big( t^{-d} \int_{B(0,t)} |\Delta_h^N (f - S^M f)(x)| \, dh\Big) ^q \frac{dt}{t}\Big)^{\frac{1}{q}}\Big| \mathcal{M}^u_p (\R)\Big \Vert
  \\
 && \qquad \qquad \qquad \leq   c_4 \,  \varepsilon  +  \frac{1}{j} +  c_5 \,   \delta ^{\frac 1u - \frac 1p}\,
\Vert  f - S^{M} f  \vert  F^{s}_{p,q}(\R)  \Vert \le c_7 \, \varepsilon + \frac{1}{j}
\eeqq
if $M$ is chosen large enough. So if $ j $ tends to infinity this proves \eqref{di_dif_2_e3}. The proof is complete.
\end{proof}

To continue we have to deal with the following subspaces of $\ce^s_{u,p,q} (\R)$.

\begin{defi}
Let $1\le p < u <\infty$, $1\le  q < \infty$ and $s>0$.
Let $B$ be a ball in $\R$. Then $\ce^s_{u,p,q} (\R;B)$ is the collection of all $f\in \ce^s_{u,p,q}(\R)$ satisfying $\supp f \subset B$.
\end{defi}

Putting together Lemma \ref{lem1} and \ref{lem2} we obtain the following theorem which will be our main tool for what follows.

\begin{satz}\label{good2}
 Let $1\le p < u <\infty$, $1\le  q < \infty$, $s>0$ and   let $B$ be a ball in $\R$.
Then
$f \in {\ce}^{s}_{u,p,q} (\R;B)$  belongs to
$ \accentset{\diamond}{\ce}^{s}_{u,p,q} (\R)$
if and only if \eqref{ws-21} and \eqref{ws-22} hold uniformly in $y\in \R$.
\end{satz}


\subsection{On the intersection of Lizorkin-Triebel-Morrey spaces}


Intersections of Lizorkin-Triebel-Morrey spaces will play a role in the description
of the interpolation spaces, see Section \ref{s-interpolation}.
In particular we are interested in properties of
${\ce}^{s_0}_{u_0,p_0,q_0} (\R) \cap {\ce}^{s_1}_{u_1,p_1,q_1} (\R)$.

\begin{lem}\label{step1}
Let $\tz\in(0,1)$, $s_i\in\rr$,
$p_i \in [1,\infty)$, $q_i\in[1,\fz]$ and $u_i\in[p_i,\fz)$ with $i\in\{0,1\}$
such that $s=(1-\tz)s_0+\tz s_1$,
\[
\frac1p=\frac{1-\tz}{p_0}+\frac\tz{p_1}\, , \quad \frac1q=\frac{1-\tz}{q_0}+\frac\tz{q_1}\, \quad
\mbox{and} \quad \frac1u=\frac{1-\tz}{u_0}+\frac{\tz}{u_1}\, .
\]
Then we have
\[
{\ce}^{s_0}_{u_0,p_0,q_0} (\R) \cap {\ce}^{s_1}_{u_1,p_1,q_1} (\R) \hookrightarrow
{\ce}^{s}_{u,p,q} (\R) \, .
\]
\end{lem}

\begin{proof}
Because of our assumptions and H\"older's inequality we have
\[
 \|\, (2^{js}\, a_j)_{j=0}^\infty|\ell_q\|\le
  \|\, (2^{js_0}\, a_j)_{j=0}^\infty|\ell_{q_0}\|^{1-\Theta}\,
  \|\, (2^{js_1}\, a_j)_{j=0}^\infty|\ell_{q_1}\|^\Theta\, .
\]
This will be applied with $a_j := \cfi [\varphi_j \cf f]$ and $j \in \mathbb{N}_{0}$.
We continue by a further application of  H\"older's inequality and find
\beqq
\Big\| \, \|\,
&& \hspace{-0.5cm}(2^{js}\, a_j)_{j=0}^\infty|\ell_q\|\, \Big|L_p (B(y,r))\Big\|
\\
& \le &
\Big\| \,  \|\, (2^{js_0}\, a_j)_{j=0}^\infty|\ell_{q_0}\|^{1-\Theta}\,
  \|\, (2^{js_1}\, a_j)_{j=0}^\infty|\ell_{q_1}\|^\Theta\, \Big|L_p (B(y,r))\Big\|
\\
& \le &
\Big\| \,  \|\, (2^{js_0}\, a_j)_{j=0}^\infty|\ell_{q_0}\|\, \Big | L_{p_0}(B(y,r))\Big\|^{1-\Theta}\,\Big\|\,  \|\, (2^{js_1}\, a_j)_{j=0}^\infty|\ell_{q_1}\|\, \Big|L_{p_1} (B(y,r))\Big\|^\Theta\, ,
  \eeqq
which proves the claim.
\end{proof}

We need to improve Lemma \ref{step1}. Therefore we have to accept stronger restrictions.

\begin{lem}\label{step2}
Let $\tz\in(0,1)$, $0\le s_0 \le s_1$,
$1\le p_0 < p_1 <\infty $, $1 \le q_0,q_1 \le \fz$, $\min(q_0,q_1) <\infty$, $p_0 < u_0$, $p_1 < u_1 $ and
$u_0 < u_1$, such that $s=(1-\tz)s_0+\tz s_1$,
\[
\frac1p=\frac{1-\tz}{p_0}+\frac\tz{p_1}\, , \quad \frac1q=\frac{1-\tz}{q_0}+\frac\tz{q_1}\, \quad
\mbox{and} \quad \frac1u=\frac{1-\tz}{u_0}+\frac{\tz}{u_1}\, .
\]
In addition we assume
either $s_0 < s_1$ or $0 < s_0 = s_1$ and $q_1\le q_0$. Let $B$ be a ball in $\R$. Then we have
\[
{\ce}^{s_0}_{u_0,p_0,q_0} (\R;B) \cap {\ce}^{s_1}_{u_1,p_1,q_1} (\R;B) \hookrightarrow
\accentset{\diamond}{\ce}^{s}_{u,p,q} (\R) \, .
\]
\end{lem}

\begin{proof}
By Lemma \ref{step1} we already know that
\[
 {\ce}^{s_0}_{u_0,p_0,q_0} (\R) \cap {\ce}^{s_1}_{u_1,p_1,q_1} (\R)
 \hookrightarrow {\ce}^{s}_{u,p,q} (\R)\, .
\]
Now we want to employ Theorem \ref{good2}. This is possible because we have $s>0$ and $q<\infty$. Let $ f \in  {\ce}^{s_0}_{u_0,p_0,q_0} (\R;B) \cap {\ce}^{s_1}_{u_1,p_1,q_1} (\R;B)    $.
Using $p_0 < p < p_1 $ and  H\"older's inequality we find
\begin{eqnarray}\label{2.36x1}
 |B(y,r)|^{\frac 1u - \frac 1p} && \lf[\int_{B(y,r)} |f(x)|^p dx\r]^{\frac{1}{p}}
\le  |B(y,r)|^{\frac 1u - \frac{1}{p_1}} \lf[\int_{B(y,r)} |f(x)|^{p_1} dx\r]^{\frac{1}{p_1}}\nonumber
\\
 && = |B(y,r)|^{\frac 1u - \frac{1}{u_1}}
|B(y,r)|^{\frac{1}{u_1} - \frac{1}{p_1}}
\lf[\int_{B(y,r)} |f(x)|^{p_1} dx\r]^{\frac{1}{p_1}}\nonumber
\\
 && \le  |B(y,r)|^{\frac 1u - \frac{1}{u_1}}
\|\, f \, \vert  \cm_{p_1}^{u_1}(\rn) \| \, ,
\end{eqnarray}
which tends to zero if $r\to 0$ due to $u_0 < u < u_1$.
Now we proceed similarly with the term $I(f,y,r,s,u,p,q)$
given by
\beqq
&& \hspace{- 1 cm} I(f,y,r,s,u,p,q):=
\\
&& |B(y,r)|^{\frac 1u - \frac 1p}\,
\bigg[\int_{B(y,r)} \bigg(
  \int_0^1 t^{-sq} \Big( t^{-d} \int_{B(0,t)} |\Delta_h^N f(x)| \, dh\Big) ^q \frac{dt}{t}\bigg)^{\frac{p}{q}}
 dx\bigg]^{\frac{1}{p}}\,
\eeqq
with $ N > s  $. We observe
\beq\label{ws-34}
&& \hspace{- 1 cm} I(f,y,r,s,u,p,q)  \nonumber \\
&& \le   |B(y,r)|^{\frac 1u - \frac{1}{p_1}} \,
\bigg\|\bigg(
  \int_0^1 t^{-sq} \Big( t^{-d} \int_{B(0,t)} |\Delta_h^N f(x)| \, dh\Big) ^q \frac{dt}{t}\bigg)^{\frac{1}{q}}\bigg| L_{p_1} (\R)\bigg\|
\nonumber \\
&& \le
|B(y,r)|^{\frac 1u - \frac{1}{u_1}} \, \bigg\|\bigg(
  \int_0^1 t^{-sq} \Big( t^{-d} \int_{B(0,t)} |\Delta_h^N f(x)| \, dh\Big) ^q \frac{dt}{t}\bigg)^{\frac{1}{q}}\bigg| \cm_{p_1}^{u_1} (\R)\bigg\|
\nonumber
  \\
&& \le  c_1 \,  |B(y,r)|^{\frac 1u - \frac{1}{u_1}} \, \| \, f\, |\ce^{s}_{u_1,p_1,q}(\R)\|  \nonumber
\\
&& \le   c_2\,  |B(y,r)|^{\frac 1u - \frac{1}{u_1}} \, \| \, f\, |\ce^{s_1}_{u_1,p_1,q_1}(\R)\| \, ,
\eeq
where we used Proposition \ref{diff} and
the elementary embedding $\ce^{s_1}_{u_1,p_1,q_1}(\R) \hookrightarrow \ce^{s}_{u_1,p_1,q}(\R)$, see Proposition 2.1 in \cite{ysy}.
As in \eqref{2.36x1} it is obvious that
the right-hand side tends to zero for $r\to 0$ uniformly in $y$.
Hence, by Theorem \ref{good2}, \eqref{2.36x1} and  \eqref{ws-34} we finally proved
$f\in \accentset{\diamond}{\ce}^{s}_{u,p,q} (\R)$.
\end{proof}


\subsection{Some test functions}
\label{test}


In this subsection we shall investigate some families of test functions. There are two reasons for doing that.
On the one hand it allows to get a feeling for the spaces under consideration. On the other hand these families will be used in the proofs of the Propositions \ref{main3} and \ref{main4}. It is well-known that the function $f(x):= |x|^{-d/u}$, $x\in \R\setminus\{0\}$,
is an extremal function for $\cm_p^u(\R)$.
Here we shall deal with a few modifications of this extremal function.
Within this subsection we shall work with a smooth cut-off function $\psi$ supported around the origin.
More exactly, $\psi \in C_0^\infty (\R)$, radial-symmetric, real-valued, $0 \le \psi (x)\le 1$ for all $x$, $\psi (x)=1$
if $|x|\le 1$ and $\psi (x)=0$ if $|x|\ge 3/2$.

\begin{lem}\label{count_bound1}
Let $1\le p < u <\infty$ and  $1\le q \le \infty$.
\\
{\rm (i)}
Then the function
\begin{equation}\label{count_1}
h_u(x) : = ( 1 - \psi(x))\,  |x|^{- \frac{d}{u}}\, , \qquad x\in \R\, ,
\end{equation}
belongs to $ \accentset{\diamond}W^m \cm_p^u (\R)$ for all $m \in \N$.
\\
{\rm (ii)} For all $s \in \re $ we have
$h_u \in  \accentset{\diamond}\ce_{u,p,q}^{s}(\R)$.
\end{lem}

\begin{proof}
We will concentrate on (ii). Temporarily we assume $s>0$.
Clearly, $h_u$ is a $C^\infty (\R)$ function.
Let $\alpha \in \N_0^d$ be a multi-index. Then we claim
that $D^\alpha h_u $ belongs to $\ce_{u,p,q}^{s}(\R)$ for all $\alpha$, i.e.,
we claim that $h_u \in E_{u,p,q}^{s}(\R)\subset \accentset{\diamond} \ce_{u,p,q}^{s}(\R)$,
see Definition \ref{diamond7}. We shall work with Proposition \ref{diff}. Therefore  we have to deal with
\begin{equation}\label{count_2}
\|\, D^{\alpha}h_u \, |\cm_{p}^u (\R)\|
\end{equation}
and
\begin{equation}\label{count_3}
\bigg\|\bigg(
  \int_0^1 t^{-sq} \Big( t^{-d} \int_{B(0,t)} |\Delta_h^N D^{\alpha}h_u(x)| \, dh\Big) ^q \frac{dt}{t}\bigg)^{\frac{1}{q}}\bigg| \cm^u_p (\R)\bigg\|
\end{equation}
with   $ N > s$. Let us start with \eqref{count_2}. Since $h_u$ is smooth, estimates with respect to   small balls are no problem.
By means of the radial symmetry and $h_u \in C^\infty (\R)$ elementary calculations show that
it will be sufficient to estimate
\begin{align*}
I_{1} = \sup_{r > 2} |B(0,r)|^{\frac{1}{u}-\frac{1}{p}} \Big ( \int_{2 < |x| < r } \vert  D^{\alpha}  |x|^{- \frac{d}{u}}  \vert^{p} dx    \Big )^{\frac{1}{p}}.
\end{align*}
By induction one can prove for any $\alpha$ the existence of a constant $C_\alpha$
such that
\be\label{wss-11}
\vert D^{\alpha}  |x|^{- \frac{d}{u}} \vert \leq C_\alpha \,
\vert x \vert^{-\frac{d}{u} - |\alpha|}\, , \qquad |x|>0\, .
\ee
Hence
\[
I_{1}  \leq c_{1}\,  \sup_{r > 2} r^{\frac{d}{u}-\frac{d}{p}} \Big ( \int_{2 < t < r }   t^{-\frac{dp}{u} - \vert \alpha \vert p}  t^{d-1}   dt    \Big )^{\frac{1}{p}}
 < \infty.
\]
Therefore \eqref{count_2} is finite.
\\
Now we turn to \eqref{count_3}.
Because of $ h_u $ is radial symmetric,  $ h_u \in C^{\infty}(\R)  $ and $\supp h_u \cap B(0,1)= \emptyset$  again some elementary calculations show that it will be sufficient to
estimate
\begin{align*}
I_{2} = \sup_{r > 2 + N} r^{\frac{d}{u}-\frac{d}{p}} \Big ( \int_{2 + N < |x| < r }  \bigg( \int_0^1 t^{-sq} \Big( t^{-d} \int_{B(0,t)} |\Delta_h^N D^{\alpha} |x|^{- \frac{d}{u}}
| \, dh\Big) ^q \frac{dt}{t}\bigg)^{\frac{p}{q}} dx   \Big )^{\frac{1}{p}}.
\end{align*}
Therefore we can apply a consequence of the Mean Value Theorem consisting in
\be\label{wss-12}
\vert \Delta^{N}_{h} D^{\alpha} |x|^{- \frac{d}{u}}   \vert \leq C_{\alpha,N} \vert h \vert^{N} \max_{\vert \gamma \vert = N} \sup_{\vert x - z \vert \leq N \vert h \vert } \vert D^{\gamma} D^{\alpha} |z|^{- \frac{d}{u}}   \vert
\ee
for some constant $C_{\alpha,N}$ independent of $x$ with $|x|> 2+N$ and $h$ with $|h|<1$.
Using this and $s<N$ we obtain
\beqq
I_{2} & \leq &  c_{2} \, \sup_{r > 2 + N} r^{\frac{d}{u}-\frac{d}{p}}
\Big ( \int_{2 + N < |x| < r }  \bigg( \int_0^1 t^{-sq+Nq}   \max_{\vert \gamma \vert = N} \sup_{\vert x - z \vert \leq N  } \vert D^{\gamma + \alpha} |z|^{- \frac{d}{u}}   \vert^q   \frac{dt}{t}\bigg)^{\frac{p}{q}}     dx    \Big )^{\frac{1}{p}}
\\
& \leq & c_{3} \, \sup_{r > 2 + N} r^{\frac{d}{u}-\frac{d}{p}}
\Big ( \int_{2 + N < |x| < r }    \max_{\vert \gamma \vert = N} \sup_{\vert x - z \vert \leq N  }   |z|^{- \frac{dp}{u}- p| \gamma| - p|\alpha |}     \bigg( \int_0^1 t^{-sq+Nq}     \frac{dt}{t}\bigg)^{\frac{p}{q}}     dx    \Big )^{\frac{1}{p}}
\\
& \leq & c_{4}\,  \sup_{r > 2 + N} r^{\frac{d}{u}-\frac{d}{p}} \Big ( \int_{2 + N < |x| < r }    \sup_{\vert x - z \vert \leq N  }   |z|^{- \frac{dp}{u}- pN - p|\alpha |}      dx    \Big )^{\frac{1}{p}}.
\eeqq
For $ 2 + N < |x| < r   $ we define $ z' := \frac{x(|x|-N)}{|x|}   $. Then because of
\begin{align*}
| z' | = |x| - N \qquad \qquad \mbox{and} \qquad \qquad |z'-x| = N
\end{align*}
we obtain
\begin{align*}
\sup_{\vert x - z \vert \leq N  }   |z|^{- \frac{dp}{u}- pN - p|\alpha |} =  \Big | \frac{x(|x|-N)}{|x|} \Big  |^{- \frac{dp}{u}- pN - p|\alpha |} =  (|x|-N)^{- \frac{dp}{u}- pN - p|\alpha |}.
\end{align*}
Now we insert this in our estimate and use $ \vert \alpha \vert = M \in \mathbb{N}_{0}$ to find
\begin{align*}
I_{2} & \leq c_{4}\,  \sup_{r > 2 + N} r^{\frac{d}{u}-\frac{d}{p}} \Big ( \int_{2 + N < |x| < r }    (|x|-N)^{- \frac{dp}{u}- pN - pM}     dx    \Big )^{\frac{1}{p}}
\\
& \leq c_{4}\,  \sup_{r > 2 } r^{\frac{d}{u}-\frac{d}{p}} \Big ( \int_{2  < |x| < r }    |x|^{- \frac{dp}{u}- pN - pM}     dx    \Big )^{\frac{1}{p}}
\\
& \leq c_{5}\,  \sup_{r > 2 } r^{\frac{d}{u}-\frac{d}{p}} \Big ( \int_{2  < t < r }    t^{- \frac{dp}{u}- pN - pM}  t^{d-1}   dt    \Big )^{\frac{1}{p}} \, .
\end{align*}
But this term is almost the same as in the estimate of $  I_{1} $. So like before we
find $ I_{2} < \infty   $. This proves the claim with $s>0$.
In the case $s\le 0$ we may use the continuous embedding \\
$\accentset{\diamond}\ce_{u,p,q}^{s_0}(\R) \hookrightarrow  \accentset{\diamond}\ce_{u,p,q}^{s_1}(\R)$ with $s_1 < s_0$.
\\
Observe that (i) follows from \eqref{count_2}.
\end{proof}

Next we consider the function
\begin{equation}\label{0619b}
f_\az(x):=\psi(x)\, |x|^{-\az},\qquad x\in \R\setminus \{0\}\, , \quad \az >0\, .
\end{equation}
Already in \cite{ysy3}, page 1849, one can find
that in case $1\le p <u <\infty$ we have
 $f_\az \in \cm_p^u(\rn)$ if and only if $\az\le \frac du$.
Moreover we have  $f_{d/u} \notin \accentset{\diamond} \cm_p^u(\rn)$.

As a consequence of  Lemma \ref{morrey43}
one obtains a characterization of $\accentset{\diamond} W^m\cm_p^u(\rn)$.

\begin{lem}\label{stern}
Let $1\le p<u<\fz$ and $m\in\nn$. Then
$\accentset{\diamond} W^m\cm_p^u(\rn)$ is equal to the collection of all
$f\in W^m\cm_p^u(\rn)$ such that, for any $\beta\in\N_0^d$ with $|\beta|\le m$,
\begin{equation}\label{0611a}
\lim_{r \downarrow 0}
|B(y,r)|^{\frac{1}{u}- \frac{1}{p}}\lf[\int_{B(y,r)} |D^\beta f(x)|^p\,dx\r]^{\frac{1}{p}}   =  0
\end{equation}
uniformly in $y\in\rn$.
\end{lem}

Now it is easy to check the regularity of $f_\alpha$ with respect
to the scale  $\accentset{\diamond} W^m\cm_p^u(\rn)$.

\begin{lem}\label{lem0611}
Let $1\le p<u<\fz$, $m\in\nn$ and $m<\frac du$. Then

{\rm(i)} $f_\az\in W^m \cm_p^u(\rn)$ if and only if $m+\az\le \frac du$;

{\rm(ii)}  $f_\az\notin \accentset{\diamond} W^m \cm_p^u(\rn)$ if  $m+\az = \frac du$.
\end{lem}

\begin{proof}
{\em Step 1.} Proof of (i). Let $\beta\in\N_0^d$ with $|\beta|\le m$. It follows from the Leibniz rule, \eqref{wss-11} and the smoothness of $\psi$ that
\[
|D^\beta(f_\az)(x)|\le C_{\alpha,\beta} \,  |x|^{-(\az+|\beta|)}\, ,\qquad |x|< \frac 32\, ,
\]
with an appropriate constant $C_{\alpha, \beta}$. Hence with $m+\az\le \frac du$ we find $D^\beta f_\az\in \cm_p^u (\R)$ and therefore $f_\az \in W^m \cm_p^u(\rn)$.

Conversely, let $f_\az\in W^m\cm_p^u(\rn)$. We fix $\beta:=(m,0,\cdots,0)$.
We need to distinguish $m$ even and $m$ odd.
If $m=  2m'$, then
\[
D^\beta(f_\az)(x)=D^\beta(|x|^{-\az})=
\sum_{i=0}^{m'} \, c_i \frac{x_1^{2i}}{|x|^{\az+2m'+ 2i}},\qquad |x|<1\, ,
\]
where $\{c_i\}_{i=0}^{m'}$ are appropriate constants independent of $x$.
If  $m= 2m'+1$, then
\[
D^\beta(f_\az)(x)=D^\beta(|x|^{-\az})=
\sum_{i=0}^{m'} \, d_i \frac{x_1^{2i+1}}{|x|^{\az+2m'+ 1+  2i}},
\]
where $\{d_i\}_{i=0}^{m'}$ are appropriate constants independent of $x$.
Observe that the terms $\frac{x_1^{j}}{|x|^{\az+2m'+ j}}$ are ordered, i.e.,
\[
\frac{|x_1|^{j+2}}{|x|^{\az+2m'+ j+2}} \le \frac{|x_1|^{j}}{|x|^{\az+2m'+ j}}\, .
\]
Now we choose a subset $A$ of $\R$  and a constant $c>0$ by
\[
A:=\Big\{x\in\rn:~~|x|<1\, , ~~ |x_1|\ge \frac{\max (|x_2|,\cdots,|x_d|)}{c} \Big\}\, .
\]
Let $E$ denote the minimum of those constants $c_0, \ldots \, , c_{m'}, d_0, \ldots \, , d_{m'}$, which are positive.
Then  $c \ge 1$ is chosen in such a way that
\[
 \Big|\sum_{i=0}^{m'} \, c_i \frac{x_1^{2i}}{|x|^{\az+2m'+ 2i}}\Big|
 \ge \frac{E}{2}\, \frac{|x_1^{2m'}|}{|x|^{\az+4m'}}\, , \qquad x\in A\, ,
\]
if $m=2m'$ and
\[
 \Big|\sum_{i=0}^{m'} \, d_i \frac{|x_1|^{2i+1}}{|x|^{\az+2m'+ 2i + 1}}\Big|
 \ge \frac{E}{2}\, \frac{|x_1|^{2m' + 1}}{|x|^{\az+4m'+1}}\, , \qquad x\in A\, ,
\]
if $m=2m'+1$.
Then for $r\in(0,1)$ and $\beta$ as above we have
\begin{align}\label{0610b}
\|f_\az|W^m\cm_p^u(\rn)\|
&\ge |B(0,r)|^{\frac1 u -\frac 1p }\, \Big ( \int_{B(0,r) \cap A}\lf|D^\beta (|x|^{-\az})\r|^p\,dx \Big )^{\frac{1}{p}}
\nonumber
\\
&\ge E_1\, |B(0,r)|^{\frac1 u -\frac 1p }\,
\Big(\int_{B(0,r)\cap A} |x|^{-(\az+m)p}\,dx \Big)^{\frac{1}{p}}
\nonumber\\
&\ge E_2 \, r^{\frac{d}u -(\alpha + m)}
\end{align}
for appropriate positive constants $E_1,E_2$ independent of $r$.
On the one hand this yields necessity of $\alpha + m \le \frac du$ in (i), on the other hand we get $f_{\frac du -m} \not \in \accentset{\diamond} W^m \cm_p^u(\rn)$, see Lemma \ref{stern}.
\end{proof}

Now we turn to the case of fractional smoothness. This will be a little bit more technical
than the previous proof.

\begin{lem}\label{test3}
Let $ s > 0  $, $ 1 \leq p < u < \infty    $ and $ 1 \leq q \leq \infty   $.
Then we have
\begin{itemize}
\item[(i)] $ f_{\alpha} \in \mathcal{E}^{s}_{u,p,q}(\R)  $ if and only if $ \alpha + s \leq d/u  $.
\item[(ii)] $  f_{\alpha} \not \in  \accentset{\diamond}\ce_{u,p,q}^{s}(\R)   $ if $ \alpha + s = d/u   $.
\end{itemize}
\end{lem}

\begin{proof}
\textit{Step 1.} Proof of (i). We will use Proposition \ref{diff}.
\\
\textit{Substep 1.1.} Sufficiency. By means of the elementary embedding
$\ce^s_{u,p,q}(\R) \hookrightarrow \ce^s_{u,p,\infty}(\R)$ we may restrict us to the case $q<\infty$.
The membership of $f_\alpha$ in Morrey spaces is already investigated. It remains to
deal with
\begin{equation}\label{lsing_1}
\bigg\|\bigg(
  \int_0^1 t^{-sq} \Big( t^{-d} \int_{B(0,t)} |\Delta_h^N f_{\alpha}(x)  | \, dh\Big) ^q \frac{dt}{t}\bigg)^{\frac{1}{q}}\bigg| \cm^u_p (\R)\bigg\|\, ,
\end{equation}
where we assume $  \alpha + s \leq d/u   $ and $ N > s$.
Because of the compact support it will be enough to deal with small balls.
Furthermore, because of the radial symmetry, it will be sufficient to study the balls $B(0,r)$ with $0 < r <1$, i.e., we are interested in
\begin{equation}\label{lsing_2}
\sup_{0 < r < 1} r^{d(\frac{1}{u}- \frac{1}{p})} \bigg (  \int_{|x| < r}  \bigg(
  \int_0^1 t^{-sq} \Big( t^{-d} \int_{B(0,t)} |\Delta_h^N |x|^{- \alpha }  | \, dh\Big) ^q \frac{dt}{t}\bigg)^{\frac{p}{q}} dx \bigg )^{\frac{1}{p}}.
\end{equation}
We split the integral with respect to $h$ into three parts, namely (a) $|h|< |x|/(2N)$ and \\
(b) $|x|/(2N)\le |h|< 2|x|$ as well as (c) $|h|\ge 2 |x|$.
\\
{\em Case (a)}.
For $ 1 \leq l \leq N  $ we have
\begin{align*}
| x + l h | \geq |x| - l |h| \geq |x| - N |h| \geq |x| - \frac{ |x|}{2} = \frac{ |x|}{2}.
\end{align*}
The Mean Value Theorem yields
\be\label{wss-14}
\vert \Delta^{N}_{h}  |x|^{- \alpha}   \vert \leq C_{\alpha,N} \vert h \vert^{N} \max_{\vert \gamma \vert = N} \sup_{\vert x - y \vert \leq N \vert h \vert } \vert D^{\gamma}  |y|^{- \alpha}   \vert \le c_1\, |h|^N\,  |x|^{-\alpha - N} \, .
\ee
We  find
\begin{align*}
& \int_0^{|x|} t^{-sq} \Big( t^{-d} \int_{\substack{|h| < t, \\ |h|< |x|/(2N)}} |\Delta_h^N |x|^{- \alpha }  | \, dh\Big) ^q \frac{dt}{t} \\
& \qquad \leq c_{1}^q \, |x|^{-(\alpha  +N) q}\,
\int_0^{|x|} t^{-sq} \Big( t^{-d} \int_{|h| < t}  |h|^N \, dh\Big) ^q \frac{dt}{t}
\\
& \qquad \leq c_{2}\,  |x|^{-(\alpha  +N) q}  \int_0^{|x|} t^{-sq + Nq}  \frac{dt}{t} \\
& \qquad \leq c_{3}\,  |x|^{-(\alpha + s) q}\,  .
\end{align*}
In the case $ 0 < |x| < t < 1$ we use the trivial estimate
\[
|\Delta_h^N |x|^{-\alpha}|\le 2^N \, \max_{0 \le l\le N}\,  |x+lh|^{-\alpha}\le
C |x|^{-\alpha}
\]
and obtain
\beqq
 \int_{|x|}^{1} t^{-sq} \Big( t^{-d} \int_{\substack{|h| < t, \\ |h|< |x|/(2N)}} |\Delta_h^N |x|^{- \alpha }  | \, dh\Big) ^q \frac{dt}{t}
& \leq & c_{4}\,  |x|^{-\alpha q}\,  \int_{|x|}^{1} t^{-sq}  \frac{dt}{t}
\\
&  \leq & c_{5} \, |x|^{-(\alpha  + s )q }.
\eeqq
Combining both estimates in Case (a) we get
\begin{equation}\label{(lsing_5)}
\int_{0}^{1} t^{-sq} \Big( t^{-d} \int_{\substack{|h| < t, \\ |h|< |x|/(2N)}} |\Delta_h^N |x|^{- \alpha }  | \, dh\Big) ^q \frac{dt}{t} \leq c_{6} \, |x|^{-(\alpha  + s )q }.
\end{equation}
{\em Case (c)}.
Next we look at the case $ 2 |x| \leq |h| < t \leq 1    $. Here we observe
\begin{align*}
& \int_{0}^{1} t^{-sq} \Big( t^{-d} \int_{2 |x| \leq |h| < t} |\Delta_h^N |x|^{- \alpha }  | \, dh\Big) ^q \frac{dt}{t} \\
& \qquad = \int_{2 |x|}^{1} t^{-sq} \Big( t^{-d} \int_{2 |x| \leq |h| < t} |\Delta_h^N |x|^{- \alpha }  | \, dh\Big) ^q \frac{dt}{t} \\
& \qquad \leq c_{7}\,  \int_{2 |x|}^{1} t^{-sq} \Big( t^{-d} \int_{2 |x| \leq |h| < t} (  |x|^{- \alpha } + |x + Nh|^{- \alpha} ) \, dh\Big) ^q \frac{dt}{t}
\\
& \qquad \leq c_{8}\, \bigg [ \int_{2 |x|}^{1} t^{-sq} \Big( t^{-d} \int_{2 |x| \leq |h| < t}  |x + Nh|^{- \alpha}  \, dh\Big) ^q \frac{dt}{t} +  |x|^{- (\alpha + s) q }\bigg ].
\end{align*}
Now since
\begin{align*}
| x + Nh | \geq N |h| - |x| \geq N |h| - \frac{|h|}{2} \geq \frac{ (2 N - 1) |h|}{2} = c_{9}\,  |h|
\end{align*}
we obtain
\begin{align}\label{lsing_6}
& \int_{0}^{1} t^{-sq} \Big( t^{-d} \int_{2 |x| \leq |h| < t} |\Delta_h^N |x|^{- \alpha }  | \, dh\Big) ^q \frac{dt}{t}
\nonumber
\\
& \qquad \leq c_{10} \bigg [ \int_{2 |x|}^{1} t^{-sq} \Big( t^{-d} \int_{2 |x| \leq |h| < t}  |h|^{- \alpha}  \, dh\Big) ^q \frac{dt}{t} +  |x|^{- (\alpha + s) q } \bigg ]
\nonumber
\\
& \qquad \leq c_{11} \, \bigg [ |x|^{- \alpha q}\,  \int_{2 |x|}^{1} t^{-sq}  \frac{dt}{t} +
|x|^{- (\alpha + s) q }\bigg ]
\nonumber
\\
& \qquad \leq c_{12}\,  |x|^{- (\alpha + s) q} .
\end{align}
{\em Case (b).} It remains to deal with  $ |x|/(2N)\le |h|< 2|x|  $.
Temporarily we assume $2|x|<1$.
In analogy to Case (c) we find
\beqq
\int_{0}^{1} t^{-sq} && \hspace{-0.3cm} \Big( t^{-d} \int_{\substack{|h| < t, \\ |x|/(2N)\le |h|< 2|x|}} |\Delta_h^N |x|^{- \alpha }  | \, dh\Big) ^q \frac{dt}{t}
\\
&=&
\int_{|x|/(2N)}^{1} t^{-sq} \Big( t^{-d} \int_{\substack{|h| < t, \\ |x|/(2N)\le |h|< 2|x|}} |\Delta_h^N |x|^{- \alpha }  | \, dh\Big) ^q \frac{dt}{t}
\\
& \le & c_{13}\,
\int_{|x|/(2N)}^{1} t^{-sq} \Big( t^{-d} \int_{\substack{|h| < t, \\ |x|/(2N)\le |h|< 2|x|}} ( |x|^{- \alpha } + |x+ Nh|^{-\alpha}  ) \, dh\Big) ^q \frac{dt}{t} \, .
\eeqq
Clearly,
\[
\int_{|x|/(2N)}^{1} t^{-sq} \Big( t^{-d} \int_{\substack{|h| < t, \\ |x|/(2N)\le |h|< 2|x|}}  |x|^{- \alpha } \, dh\Big) ^q \frac{dt}{t}
\le c_{14}\, |x|^{-(\alpha + s)q}\,  .
\]
On the other hand we  observe that
\beqq
&& \hspace{-1,2 cm}
\int_{|x|/(2N)}^{1} t^{-sq}  \Big( t^{-d} \int_{\substack{|h| < t, \\ |x|/(2N)\le |h|< 2|x|}}  |x+Nh|^{- \alpha } \, dh\Big) ^q \frac{dt}{t}
\\
&\le &
\int_{|x|/(2N)}^{1} t^{-sq}  \Big( t^{-d} \int_{|h| < \min (t,2|x|)}
|x+Nh|^{- \alpha } \, dh\Big) ^q \frac{dt}{t}
\\
&\le &
\int_{|x|/(2N)}^{1} t^{-sq}  \Big( t^{-d} \int_{|y| < \min (|x|+Nt, (2N+1)|x|)}
|y|^{- \alpha } \, dh\Big) ^q \frac{dt}{t}
\\
&\le &
\int_{|x|/(2N)}^{1} t^{-sq}   t^{-dq} (\min (|x|+Nt, (2N+1)|x|))^{(-\alpha + d)q}  \,
\frac{dt}{t}
\\
& = &
\int_{|x|/(2N)}^{2|x|} t^{-sq}   t^{-dq} (|x|+Nt)^{(-\alpha + d)q}  \,
\frac{dt}{t} +
\int^{1}_{2|x|} t^{-sq}   t^{-dq} ((2N+1)|x|)^{(-\alpha + d)q}
\frac{dt}{t} \, ,
\eeqq
where we used $\alpha <d$.
Since
\beqq
\int_{|x|/(2N)}^{2|x|} t^{-sq}   t^{-dq} (|x|+Nt)^{(-\alpha + d)q}  \,
\frac{dt}{t}  \le  c_{15}\, |x|^{-(s+\alpha)q}
\eeqq
and
\beqq
\int^{1}_{2|x|} t^{-sq}   t^{-dq} ((2N+1)|x|)^{(-\alpha + d)q}
\frac{dt}{t} \le c_{16}\, |x|^{-(\alpha + s)q}\,
\eeqq
it follows also for the case $ |x|/(2N)\le |h|< 2|x|  $
\begin{equation}\label{(lsing_7)}
 \int_{0}^{1} t^{-sq} \Big( t^{-d} \int_{\substack{|h| < t, \\ |x|/(2N)\le |h|< 2|x|}} |\Delta_h^N |x|^{- \alpha }  | \, dh\Big) ^q \frac{dt}{t} \leq c_{17}\,  |x|^{- (\alpha + s) q} .
\end{equation}
The needed modifications for the case $|x|\le 1 < 2|x|$ are obvious.
Now we are well prepared to deal with \eqref{lsing_2}.
When we combine \eqref{(lsing_5)} - \eqref{(lsing_7)} we obtain
\begin{align*}
& \sup_{0 < r < 1} r^{d(\frac{1}{u}- \frac{1}{p})} \bigg (  \int_{|x| < r}  \bigg(
  \int_0^1 t^{-sq} \Big( t^{-d} \int_{B(0,t)} |\Delta_h^N |x|^{- \alpha }  | \, dh\Big) ^q \frac{dt}{t}\bigg)^{\frac{p}{q}} dx \bigg )^{\frac{1}{p}}
  \\
	& \qquad \leq c_{18}\,  \sup_{0 < r < 1} r^{d(\frac{1}{u}- \frac{1}{p})} \bigg (  \int_{|x| < r}  |x|^{- (\alpha + s) p}  dx \bigg )^{\frac{1}{p}} \\
& \qquad \leq c_{19}\,  \sup_{0 < r < 1} r^{d(\frac{1}{u}- \frac{1}{p})} \bigg (  \int_{0}^{r}  t^{- (\alpha + s) p} t^{d-1}  dt \bigg )^{\frac{1}{p}} .	
\end{align*}
Since $   \alpha + s \leq d/u < d/p     $ this integral exists and  we find
\begin{align*}
& \sup_{0 < r < 1} r^{d(\frac{1}{u}- \frac{1}{p})} \bigg (  \int_{|x| < r}  \bigg(
  \int_0^1 t^{-sq} \Big( t^{-d} \int_{B(0,t)} |\Delta_h^N |x|^{- \alpha }  | \, dh\Big) ^q \frac{dt}{t}\bigg)^{\frac{p}{q}} dx \bigg )^{\frac{1}{p}}
  \\
& \qquad \leq c_{20}\,  \sup_{0 < r < 1} r^{d(\frac{1}{u}- \frac{1}{p})}    r^{- (\alpha + s) } r^{\frac{d}{p}}
\\
& \qquad \leq c_{21} \sup_{0 < r < 1} r^{\frac{d}{u} - \alpha - s} < \infty .
\end{align*}
This proves $ f_{\alpha} \in \mathcal{E}^{s}_{u,p,q}(\R)  $ in the case $ \alpha + s \leq d/u  $.\\
\textit{Substep 1.2.} Necessity.
Let $ \alpha + s > d/u  $. By means of the elementary embedding
$\ce^s_{u,p,q}(\R) \hookrightarrow \ce^s_{u,p,\infty}(\R) $ it will be enough to consider
the case $q= \infty$.  We claim that
\begin{equation}\label{lsing_8}
\sup_{0 < r < 1} r^{d(\frac{1}{u}- \frac{1}{p})} \bigg (  \int_{|x| < r}  \bigg(
  \sup_{0< t < 1}\,  t^{-s}
t^{-d} \int_{B(0,t)} |\Delta_h^N |x|^{- \alpha }  | \, dh\bigg)^{p} dx \bigg )^{\frac{1}{p}} = \infty\, .
\end{equation}
Write $ x = (x_{1}, x_{2}, \ldots , x_{d}) \in \R   $ and $ h = (h_{1}, h_{2}, \ldots , h_{d}) \in \R   $.
We put
\[
 \Omega (t) := \{h \in \R: ~ |h|< t\, , ~~ \frac{t}{2\sqrt{d}} \le  \min_k h_k \}\, , \qquad t>0\, .
\]
Then, for all $  0 < t \leq 1 $, it follows the existence of a positive constant $C $ such that
\be\label{wss-15}
 |\Omega (t)|\ge C \, t^{d} \, .
\ee
Let  $2^{-j-1} \le  |x| \le  2^{-j}$ for some $j \in \N$ and $\min_k x_k \ge 0$.
Moreover we assume $ 2^{-i} \leq t < 2^{-i + 1}    $ for some $ i \in \mathbb{N}  $ with $1 \le i < j-L'$, where  $L' \in \mathbb{N} $ will be chosen later.
Now let $h \in \Omega ( 2^{-i}) $. Then for $ k \in \{ 1, 2, \ldots , d   \}  $, because of $ j - i  > L' $, we observe
\[
 2 x_k h_k \ge  x_k \frac{ 2^{-i}}{ \sqrt{d}} \ge  x_k^2 \, \frac{ 2^{j-i}}{ \sqrt{d}}
 \ge x_k^2 \, 2^{L'} \, \frac{1}{ \sqrt{d}}\, .
\]
We choose $L \in \N$ such that $2^{L'} \, \frac{1}{ \sqrt{d}} \ge 2^L$. Hence
\[
 (x_k + h_k)^2 = x_{k}^{2} + 2 x_{k} h_{k} + h_{k}^2 \ge 2x_k h_k \ge 2^L x_k^2
\]
and therefore
\[
 |x+h|^\alpha \ge 2^{\alpha L/2} |x|^\alpha \, .
\]
The restrictions  $x_k, h_k \ge 0$ for all $ k \in \{ 1, 2, \ldots , d   \} $
also imply
\[
 |x+\ell h|^\alpha \ge  |x+h|^\alpha  \geq 2^{\alpha L/2} |x|^\alpha \, , \qquad \ell \in \{ 1, \ldots \, , N \} ,
\]
which results in
\be
  |x|^{-\alpha} \ge 2^{\alpha L/2} |x+ \ell h|^{-\alpha} \, , \qquad \ell \in \{ 1, \ldots \, , N \} .
\ee
Now we are able to find an appropriate estimate of $|\Delta_h^N f_\alpha|$.
Under the constraints collected above  we obtain
\beqq
|\Delta_h^N f_\alpha (x)|& \ge & |x|^{-\alpha} - \Big(\sum_{\ell=0}^{N-1} \binom{N}{\ell} |x+ (N-\ell)h|^{-\alpha} \Big)
\\
&\ge & |x|^{-\alpha} - 2^N |x+h|^{-\alpha} \\
&\ge & |x|^{-\alpha} - 2^N 2^{- \alpha L/2} |x|^{-\alpha} \\
& = & |x|^{-\alpha} ( 1 -  2^{N - \alpha L/2} ).
\eeqq
Now we choose $L \in \mathbb{N} $ as small as possible such that $ 1 -  2^{N - \alpha L/2} \geq 1/2  $ is fulfilled. Then $ L $ only depends on $ N $ and $ \alpha  $ and we get
\begin{align*}
|\Delta_h^N f_\alpha (x)| \geq \frac 12 \, |x|^{-\alpha}.
\end{align*}
Choosing $  L' \in \mathbb{N} $ as the smallest number that fulfills
$2^{L'} \, \frac{1}{ \sqrt{d}} \ge 2^L$ we derive with
\[
2^{-j-1} \le  |x| \le  2^{-j} < 2^{- 2 - L'}  \, ,  \qquad \min_k x_k \ge 0 \, ,
\]
that
\beqq
 \sup_{0< t < 1}\,  t^{-s}
t^{-d} \int_{B(0,t)} |\Delta_h^N |x|^{- \alpha }  | \, dh
& \ge &   \sup_{i \in \mathbb{N} }\,   2^{i(s+d)}
\int_{B(0,2^{-i})} |\Delta_h^N |x|^{- \alpha }  | \, dh
\\
&  \geq &   \sup_{i \in \{ 1, 2, \ldots ,j - L' - 1 \} }  \,
2^{i(s+d)} \int_{\Omega ( 2^{-i})} |\Delta_h^N |x|^{- \alpha }  | \, dh
\\
&\geq &  c_{22} \, 2^{(j-L'-1)s} \,  |x|^{- \alpha}
\eeqq
for some positive $c_{22}$ (independent of $x$) by
taking into account \eqref{wss-15}.
Next since $ 2^{-j-1} \le  |x| \le  2^{-j}    $ this can be rewritten as
\be\label{wss-16}
 \sup_{0< t < 1}\,  t^{-s}
t^{-d} \int_{B(0,t)} |\Delta_h^N |x|^{- \alpha }  | \, dh
\ge  c_{23} \,  |x|^{- (\alpha+s)}\,  .
\ee
We need one more notation
\[
 B_j^+:=  \Big(B(0, 2^{-j}) \setminus B(0 , 2^{-j-1})\Big) \cap \Big\{ x \in \R:~~ x_{k} \geq 0 \;
\mbox{for all} \; k = 1, 2, \ldots , d  \Big\}, \qquad j \in \N\, .
\]
By construction and \eqref{wss-16} it follows
\begin{align*}
 \sup_{0 < r < 1} r^{d(\frac{1}{u}- \frac{1}{p})}  &\bigg (  \int_{|x| < r}  \bigg(
  \sup_{0<t<1} \, t^{-(s+d)} \, \int_{B(0,t)} |\Delta_h^N |x|^{- \alpha }  | \, dh
\bigg)^{p} dx \bigg )^{\frac{1}{p}}
\\
& \geq c_{23}\,  \sup_{0 < r < 1} r^{d(\frac{1}{u}- \frac{1}{p})} \bigg (
\sum_{j = L' + 2}^{\infty}  \int_{B(0,r) \cap B_j^+}
   |x|^{- (\alpha + s) p }  dx \bigg )^{\frac{1}{p}}
\\
& \geq c_{24}\,  \sup_{0 < r < 1} r^{d(\frac{1}{u}- \frac{1}{p})} \bigg (  \int_{0 }^{\min(r , 2^{-L' - 2})}
   t^{- (\alpha + s) p } t^{d-1}  dt \bigg )^{\frac{1}{p}} .
\end{align*}
Now there are two possibilities. Either this integral is infinite or it is finite.
In the first case our claim follows.
In the second case, i.e., if $ -(\alpha + s) + d/p > 0 $,  we conclude
\begin{align}\label{wss-17}
& \sup_{0 < r < 1} r^{d(\frac{1}{u}- \frac{1}{p})}  \bigg (  \int_{|x| < r}  \bigg(
  \sup_{0<t<1} \, t^{-(s+d)} \, \int_{B(0,t)} |\Delta_h^N |x|^{- \alpha }  | \, dh
\bigg)^{p} dx \bigg )^{\frac{1}{p}}
\nonumber
\\
& \qquad \geq c_{25}\,  \sup_{0 < r < 1}\,  r^{d(\frac{1}{u}- \frac{1}{p})}
   \min(r , 2^{-L' - 2})^{- (\alpha + s)  + \frac{d}{p} }
\nonumber
\\
& \qquad \geq c_{26}\,  \sup_{0 < r < 2^{-L' - 2}}\,  r^{\frac{d}{u}- \alpha - s   } \, .
\end{align}
Because of  $ \alpha + s > d/u  $ the right-hand side is not finite and therefore we have $  f_{\alpha} \not \in \mathcal{E}^{s}_{u,p,q}(\R)$.
\\
\textit{Step 2}. Proof of (ii).
We fix $ \alpha: = \frac du-s>0$. By means of the elementary embedding
$\accentset{\diamond}\ce^s_{u,p,q}(\R) \hookrightarrow \accentset{\diamond}\ce^s_{u,p,\infty}(\R)$
it will be enough to concentrate on $q= \infty$.
Here we can apply Step 1.2, in particular \eqref{wss-17}.
It follows
\beqq
 \sup_{0 < r < 1}\, && \hspace{-0.2cm}   r^{d(\frac{1}{u}- \frac{1}{p})}  \bigg (  \int_{|x| < r}  \bigg(
  \sup_{0<t<1} \, t^{-(s+d)} \, \int_{B(0,t)} |\Delta_h^N |x|^{- \alpha }  | \, dh
\bigg)^{p} dx \bigg )^{\frac{1}{p}}
\\
& \geq &  c_{26}\,  \sup_{0 < r < 2^{-L' - 2}}\,  r^{\frac{d}{u}- \alpha - s} = c_{27} >0
\eeqq
and hence, by Lemma \ref{lem1} we find  $ f_{\frac du - s} \not \in \accentset{\diamond}\ce_{u,p,q}^{s}(\R)$.
The proof is complete.
\end{proof}

\begin{rem}
 \rm
The regularity of the functions $f_\alpha$, defined in \eqref{0619b}, attracted some attention in the literature.
Membership in $\ce^s_{p,p,q}(\R)= F^s_{p,q} (\R)$ has been investigated in
\cite[Lemma~2.3.1/1]{rs96}. However, there a totally different method has been applied.
In \cite[p.~97]{rs96} additional references are given.
\end{rem}


\section{Extension operators for the spaces ${\ce}^{s}_{u,p,q} (\Omega)$}
\label{slava}


The main subject of this section  will be Rychkov's extension operator, see \cite{ry99}.
To adopt our approach to what has been done in \cite{ry99}, we mention  the following
generalization of our Definition \ref{LTMS}.
First we need some further notation.
For any function $h$, we use $L_h\in \N_0$ to denote the maximal number such that $h$ has
vanishing moments up to order $L_h$, namely,
\[
\int_{\R} x^\az h(x)\,dx=0 \qquad \mbox{for all multi-indices $\az$ with $|\az|\le L_h$.}
\]
If either no or all moments vanish, then we put $L_h=-1$ or
$L_h=\fz$, respectively.
For a given function $\lambda$ we define  $\lambda_j(x):= 2^{jd}\, \lambda (2^jx)$ with $x\in \R$ and $j\in \N$.

\begin{lem}\label{ry1}
Let $1 \le p \le u< \infty $, $1\le q \le \infty$ and $s\in \re$.
Let $\lambda_0 \in \cs (\R)$ be a function such that
\beq\label{system1}
&&
\int_{\R} \lambda_0 (x)\, dx  \neq   0\, ,
\\
\label{system2}
&&
L_\lambda  \ge  [s]\, , \quad \mbox{where}\quad
\lambda (\cdot):=\lambda_0 (\cdot)-2^{-d}\, \lambda_0(\cdot/2)\, .
\eeq
Then the Lizorkin-Triebel-Morrey space  $ \ce^{s}_{u,p,q}(\re^d)$ is the
collection of all tempered distributions $f \in \mathcal{S}'(\R)$
such that
\beqq
\|\, f \, |\ce^s_{u,p,q}(\R)\|_{\lambda_0} =
 \bigg\| \Big ( \sum\limits_{k=0}^{\infty} 2^{k s q}\, |\cfi[\lambda_{k}  \cf f](\, \cdot \, )|^q\Big)^{\frac{1}{q}}
\bigg| \cm^u_p (\re^d)\bigg\| <\infty
\eeqq
in the sense of equivalent norms
(with usual modification if $q= \infty$).
\end{lem}

\begin{proof}
 In principle the proof follows the same lines as in case $p=u$, see  \cite{bpt96}, \cite{bpt97}
and \cite[Prop.~1.2]{ry99}. So we skip the details.
\end{proof}


For a special Lipschitz domain $\Omega$, one can find a narrow vertically
directed cone $K$ with vertex at origin that its shifts $x+K$ are
in $\Omega$ for every $x\in \Omega$. For example, we may take
$$K:=\{(x',x_d)\in\R:\ |x'|<A^{-1}x_d\},$$
where $A$ denotes the Lipschitz constant of $\omega$.
Let $-K:=\{-x:\ x\in K\}$ be the ''reflected'' cone. Then for every test function
$\gamma\in \cd(-K)$ and $f\in \cd'(\Omega)$, the convolution
$\gamma\ast f(x)=\langle f,\gamma(x-\cdot)\rangle$ is well defined in $\Omega$
since $\supp \gamma(x-\cdot)\subset\Omega$ for $x\in\Omega$.

\begin{prop}\label{p3-19}
Let $\Omega \subset \R$ be a special Lipschitz domain and let $K$ be one associated cone as above.
Let $\vz_0\in\cd (-K)$ have nonzero integral and let
$\vz(\cdot):=\vz_0(\cdot)-2^{-d}\vz_0(\cdot/2)$. Then for any given
$L\in \N_0$ there exist functions $\psi_0,\ \psi\in \cd(-K)$ such that
$L_\psi\ge L$ and
\begin{equation}\label{3-22x}
f=\sum_{j=0}^\fz \psi_j\ast \vz_j\ast f
\end{equation}
for all $f\in \cd'(\Omega)$.
\end{prop}

Proposition \ref{p3-19} was established by Rychkov in \cite{ry99}.
In the following, for any $f:\ \Omega\to\C$, denote by $f_{\Omega}$ its extension
from $\Omega$ to all of $\rn$ by zero.
In addition, if $g:~\R \to \C $, then $g_{|_\Omega}$ denotes the restriction of $g$ to $\Omega$. This notation will be also used for distributions.

\begin{satz}\label{thm3-19x}
Let $\Omega \subset \R$ be a special Lipschitz domain and  $K$ its associated cone.
Let $s\in\re$, $q\in [1,\fz]$ and $1 \le p\le u<\fz$.
Let $\vz_0\in\cd(-K)$ satisfy \eqref{system1} and \eqref{system2}.
Let $\psi_0,\ \psi\in\cd(-K)$ be given by Proposition \ref{p3-19} such that
 $L_\psi>d/\min(p,q)$. Then the map $E$
 defined by
 \begin{equation}\label{3-19x}
 E f :=  \sum_{j=0}^\fz  \psi_j\ast (\vz_j\ast f)_{\Omega}\, , \qquad f\in \cd'(\Omega)\, ,
 \end{equation}
 induces a linear and bounded  extension operator from
 $\mathcal E_{u,p,q}^s(\Omega)$ into the space $\mathcal E_{u,p,q}^s(\R)$.
 Moreover, for any $f\in\cd'(\Omega)$ we have  $ E (f)_{|_\Omega}=f$.
\end{satz}

\begin{proof}
We follow Rychkov \cite{ry99}. Only a few modifications have to be made.
The parameters $s,p,u,q,d$ are considered to be fixed in what follows.\\
{\em Step 1.}
Let $X$ be the space of all sequences
$\{g_j\}_{j\in \N_0}$ of locally integrable functions on $\rn$ such that
\[
\|\, \{g_j\}_{j=0}^ \infty\, |X\|
:=\Big   \|  \Big ( \sum_{j=0}^\infty |2^{js}G(g_j)|^q   \Big )^\frac1q \Big| \cm_p^u(\rn) \Big  \|<\fz,
\]
where $G(g_j)$ denotes the Peetre maximal function of $g_j$, namely,
$$
G(g_j)(x):=\sup_{y\in\rn}\frac{|g_j(y)|}{(1+2^j|x-y|)^N},\qquad x\in\rn \, .
$$
The natural number $N$ will be chosen such that
\be\label{wss-05}
\frac{d}{\min (p,q)} < N \le L_\psi \, .
\ee
We claim that, for any $\{g_j\}_{j=0}^\infty \in X$,
the series $\sum_{j=0}^\infty \psi_j \ast g_j$ converges in $\cs'(\rn)$
and
\begin{equation}\label{3-20x}
\Big\| \, \sum_{j=0}^\infty\, \psi_j \ast g_j\, \Big| {\ce_{u,p,q}^{s}(\rn)}\Big\|
\ls \|\, \{g_l\}_{l=0}^\infty\,  |X\| .
\end{equation}
By \cite[(2.14)]{ry99} we know that, if $L_\vz\ge [s]$ and $L_\psi\ge N$, then there
exists some $\sigma\in(0,\fz)$, such that
\begin{equation}\label{3-19z}
2^{ls}|\vz_l\ast\psi_j\ast g_j(x)| \ls 2^{-|l-j|\sigma}2^{js}G(g_j)(x)
\end{equation}
with hidden constant independent of $x\in \R$, $l,j\in \N_0$ and $\{g_j\}_{j=0}^\infty$.
By Lemma \ref{ry1} we may assume that $\ce^s_{u,p,q}(\R)$ is equipped with
the norm generated by $\varphi_0$.
Thus, for any $j\in \mathbb{N}_{0}$, we have
\begin{align*}
\|\psi_j\ast g_j \vert  \ce_{u,p,q}^{s-2\sigma}(\rn)  \|
\ls \bigg\| \lf ( \sum_{l=0}^\fz 2^{-(2l+|l-j|)\sigma q}[2^{js}G(g_j)]^q
\r ) ^\frac1q \, \bigg| {\cm_p^u(\rn)} \bigg\|\, .
\end{align*}
From this we conclude that, for any $j\in \N_0$,
\beqq
\|\, \psi_j \ast g_j \, |{\ce_{u,p,q}^{s-2\sigma}(\rn)}\|
& \ls &  2^{-j\sigma}\, \|\, 2^{js}\,  G(g_j) \vert \cm_p^u(\rn)  \|
\\
& \ls &  2^{-j\sigma} \, \|\, \{G(g_l)\}_{l=0}^\infty \, |{\cm_p^u(\ell_q^s(\rn))}\| \, .
\eeqq
This implies that, for any $k_1,\ k_2\in \nn, \,  k_1<k_2$, we find
\[
\Big\| \, \sum_{j=k_1}^{k_2}\psi_j\ast g_j\Big| {\ce_{u,p,q}^{s-2\sigma}(\rn)}\Big\|
 \ls \sum_{j=k_1}^{k_2} \, 2^{-j\sigma}\, \|\, \{G(g_l)\}_{l=0}^\infty \, |\cm_p^u(\ell_q^s(\rn))\|
 \ls 2^{-k_1 \sigma}\, .
\]
Hence,  $\sum_{j=0}^\infty \psi_j\ast g_j$ converges
in $\ce_{u,p,q}^{s-2\sigma}(\rn)$ and therefore in $\cs'(\rn)$, since
$\ce_{u,p,q}^{s-2\sigma}(\rn)\hookrightarrow \cs'(\rn)$. Now we turn to the norm estimate.
By \eqref{3-19z} we also have for any $l\in \N_0$ and any $x\in\rn$,
\[
2^{ls}\, \bigg| \vz_l\ast \Big(\sum_{j=0}^\infty \psi_j\ast g_j\Big)(x)\bigg|
\ls \sum_{j=0}^\infty 2^{-|l-j|\sigma} 2^{js}\,  G(g_j)(x)\, .
\]
Taking the $\cm^u_p (\ell_q)$-norm on both sides it is easy to see that
\eqref{3-20x} holds true.
\\
{\em Step 2.} Now we aim to prove that $f\in \ce_{u,p,q}^s(\Omega)$ implies
\begin{equation}\label{3-20y}
 E (f)\in \ce_{u,p,q}^s(\rn)\qquad \mbox{and} \qquad\| \, E (f)\, |{\ce_{u,p,q}^s(\rn)}\|
\ls \|\, f\, |{\ce_{u,p,q}^s(\Omega)}\|\, .
\end{equation}
By definition, for any $\varepsilon\in(0,\fz)$, there
exists $g\in \ce_{u,p,q}^s(\rn)$ such that $g_{|_\Omega}=f$ in $\cd'(\Omega)$
and
\begin{equation}\label{3-21x}
\|\, g \, |{\ce_{u,p,q}^s(\rn)}\| \le \|\, f\, |{\ce_{u,p,q}^s(\Omega)}\| + \varepsilon.
\end{equation}
Let $g_j:=(\vz_j\ast f)_\Omega$ with $j\in\N_0$. We will show that
\begin{equation}\label{3-20z}
\|\, \{(\vz_j\ast f)_\Omega\}_{j=0}^\infty \,|X\|\ls \|\, g\, |{\ce_{u,p,q}^s(\rn)}\| \, .
\end{equation}
Again we apply an inequality due to Rychkov \cite[p.\,248]{ry99}.
We have
\[
\sup_{y\in\rn}\, \frac{|(\vz_j\ast f)_{\Omega}(y)|}{(1+2^j|x-y|)^N}
\ls
\left\{ \begin{array}{lll}
\sup\limits_{y\in\Omega}\, \frac{|\vz_j\ast f(y)|}{(1+2^j|x-y|)^N}  & \quad & \mbox{if}\quad  x\in \Omega ;
\\
\sup\limits_{y\in\Omega} \, \frac{|\vz_j\ast f(y)|}{(1+2^j|\widetilde x-y|)^N}
&& \mbox{if}\quad  x \notin \overline{\Omega}\, .
        \end{array}
\right.
\]
Here $\widetilde{x}:=(x',2w(x')-x_d)\in \Omega$ is the point symmetric to
$x\notin \overline{\Omega}$ with respect to $\partial \Omega$.
Since the convolution of $\varphi_j$ with $f$ in  $\Omega$ is only using values in $\Omega$ we obtain
\[
\vz_j\ast f(x)=\vz_j\ast g(x) \qquad \mbox{for any $x\in\Omega$}\, .
\]
Hence
\[
\sup_{y\in\rn}\, \frac{|(\vz_j\ast f)_{\Omega}(y)|}{(1+2^j|x-y|)^N}
\ls
\left\{ \begin{array}{lll}
G(\vz_j\ast g)(x) &\quad & \mbox{if}\quad  x\in \Omega ;
\\
G(\vz_j\ast g)(\widetilde x)
&& \mbox{if}\quad  x\notin \overline{\Omega}\, .
\end{array} \right.
\]
Obviously, for any ball $B(z,r) \subset \rn$, we know that
\beqq
&& \hspace*{-0.2cm}|B(z,r)|^{\frac{1}{u}- \frac{1}{p}} \Big \|\, \Big ( \sum_{j=0}^\infty \Big|2^{js}\, \sup_{y\in\rn}
\frac{(\vz_j\ast f)_{\Omega}}{(1+2^j|\cdot-y|)^N}\Big|^q\Big )^{\frac1q}\Big | {L^p(B(z,r))}\Big\|
\\
&& \quad \ls |B(z,r)|^{\frac{1}{u}- \frac{1}{p}} \Big\|\, \Big ( \sum_{j=0}^\infty \Big|2^{js}\, G(\varphi_j \ast g)(x)\Big|^q\Big )^{\frac1q}\Big| {L^p(B(z,r) \cap \Omega)}\Big\|
\\
&& \quad + |B(z,r)|^{\frac{1}{u}- \frac{1}{p}} \Big\|\, \Big ( \sum_{j=0}^\infty \Big|2^{js}\,
G(\varphi_j \ast g)(\widetilde{x})\Big|^q\Big )^{\frac1q}\Big| {L^p(B(z,r)\cap
\overline{\Omega}^{\complement})}\Big\|
\\
&& \quad =: \mbox{I} + \mbox{II}\, .
\eeqq
Clearly,
\[
\mbox{I} \ls \|\, \{G(\vz_j\ast g)\}_{j=0}^\infty \, |{\cm_p^u(\ell_q^s(\rn))}\|\, .
\]
Concerning II we argue as follows.
Let $x\in B(z,r)\cap \overline{\Omega}^\complement$.
Independent on the situation ($z\in \Omega $ or $z\not \in \Omega$) we associate to $z$
the vector $\widetilde{z}:= (z',2\omega(z')-z_d)$. Here $\omega$ refers to the function
occuring in the definition of a special Lipschitz domain, see Definition \ref{lipdo}.
It follows that
\[
 |\widetilde{z} - \widetilde{x}|^2 \le  |z'-x'|^2 + (2 \, A \, |z'- x'|+ |z_d-x_d|)^2 <
 \max (2A, 1)^2 \, r^2\, ,
\]
i.e., $\widetilde{x} \in B(\widetilde{z}, \max (2A,1) r)$.
By Rademacher's Theorem $\omega$ is differentiable almost everywhere in $\re^{d-1}$.
Using this  we observe that the transformation
$T(x) = \widetilde{x}$ with $x\in \R$ has Jacobi determinant $ |\det J_T (x)| = 1 $ almost everywhere.
Thus, it follows from a change of variable formula, see, e.g., \cite{Haj}, \cite{BI}, that
\beqq
\int_{B(z,r)\cap \overline{\Omega}^{\complement}} && \Big ( \sum_{j=0}^\infty \Big|2^{js}\,
G(\varphi_j \ast g) (T(x))\Big|^q\Big )^{\frac pq}\, dx
\\
& \ls &
\int_{B(\widetilde{z},\max (2A,1)r)} \Big ( \sum_{j=0}^\infty \Big|2^{js}\,
G(\varphi_j \ast g) (\widetilde{x})\Big|^q\Big )^{\frac pq}\, d\widetilde{x}\, .
\eeqq
Applying this inequality we derive
\begin{eqnarray*}
{\rm II}
&\ls & |B(z,r)|^{\frac1u-\frac1p}
\bigg(\int_{B(\widetilde{z},\max (2A,1)r)}
\Big ( \sum_{j=0}^\infty \Big|2^{js}\,
G(\varphi_j \ast g) (\widetilde{x})\Big|^q\Big )^{\frac pq}\, d\widetilde{x}\bigg)^{\frac{1}{p}}
\\
&\ls & |B(\wz z,\max(2A,1)r)|^{\frac1u-\frac1p}\\
&&\quad\times\bigg(\int_{B(\widetilde{z},\max (2A,1)r)}
\Big (\sum_{j=0}^\infty \Big|2^{js}\,
G(\varphi_j \ast g) (\widetilde{x})\Big|^q\Big )^{\frac pq}\, d\widetilde{x}\bigg)^{\frac{1}{p}}
\\
&\ls & \|\, \{G(\vz_j\ast g)\}_{j=0}^\infty \, |{\cm_p^u(\ell_q^s(\R))}\| \, .
\end{eqnarray*}
From this, combined with the characterization of $\ce_{u,p,q}^s(\rn)$
via the Peetre maximal function  with $N>d/\min (p,q)$ (see, for example, \cite[Subsection 11.2]{lsuyy}),
we further deduce that ${\rm II}\ls \|g \vert \ce_{u,p,q}^s(\rn)  \|$. Thus,
\eqref{3-20z} is proved.
By Step 1, \eqref{3-20y}, and \eqref{3-20z}  we conclude that
\begin{equation*}
\| E (f) \vert \ce_{u,p,q}^s(\rn) \|
\ls \| \{(\vz_j\ast f)_{\Omega} \}_{j=0}^\infty  \vert \cm_p^u(\ell_q^s(\rn))  \|
\ls \|f \vert \ce_{u,p,q}^s(\Omega) \| +\varepsilon.
\end{equation*}
When $\varepsilon $ tends to zero, we find that $ E $ is a bounded linear
operator from $\ce_{u,p,q}^s(\Omega)$ into $\ce_{u,p,q}^s(\rn)$.
\\
{\em Step 3.} Let $\rho \in \cd (\Omega)$. Then
\[
\supp \int_{\R}\, \psi_j (x- \cdot )\rho (x)\, dx \subset \overline{\Omega}\, ,
\]
where we used that the supports of $\psi_0$ and $\psi$ are lying in $-K$.
Hence
\beqq
\int_{\R}  \Big(\int_{\R}\, \psi_j (x-y)\rho (x)\, dx\Big) && \hspace{-0.4cm}\,
(\varphi_j \ast f)_{\Omega} (y)\, dy
\\
&= &\int_{\R}  \Big(\int_{\R}\, \psi_j (x-y)\rho (x)\, dx\Big) \, (\varphi_j \ast f)(y)\, dy\, .
\eeqq
Finally, from Proposition \ref{p3-19}  we conclude that
$$ E (f)_{|_\Omega} = \sum_{j=0}^\infty \psi_j\ast\vz_j\ast f=f \qquad{\rm in}\ \cd'(\Omega).$$
This finishes the proof of Theorem \ref{thm3-19x}.
\end{proof}

We remark that the extension operator $E$ in Theorem \ref{thm3-19x}
depends on $p,\ q$ and $s$. More precisely, we need to have
\begin{equation}\label{4.1}
[s]\le L_\vz \qquad \mbox{and} \qquad \min(p,q)> \frac{d}{L_\psi}.
\end{equation}
However, Rychkov \cite{ry99} has shown how to overcome these restrictions.
He constructed an universal extension operator, i.e., an extension operator, which works for all admissible parameter constellations simultaneously.
In view of \eqref{4.1}, one is tempted to take $L_\varphi = L_\psi =\infty$, which is
certainly impossible for compactly supported functions,
but can be achieved with $\varphi, \psi$ rapidly decreasing at infinity.

Let $\Omega$ and $K$ be as above. By $\cs'(\Omega)$ we denote the subset of $\cd'(\Omega)$
consisting of all distributions having finite order and at most polynomial growth
at infinity. More precisely, $f\in \cs'(\Omega)$ if and only if  the estimate
\begin{equation}
|\langle f,\gamma\rangle |\le c\sup_{x\in\Omega,|\alpha|\le M} |D^\alpha\gamma(x)|(1+|x|)^M,
\qquad\mbox{for all}\quad \,\gamma\in\cd(\Omega),
\end{equation}
is true with some constants $c$ and $M\in \N_0$ depending on $f$.

\begin{rem}\label{rem3-23x}
\rm
By \cite[p.\,250]{ry99}, we find that $f\in \cs'(\Omega)$ if and only if there exists
a $g\in\cs'(\rn)$ such that $g_{|_\Omega}=f$. In particular, $\ce_{u,p,q}^s(\Omega)$
is a subset of $\cs'(\Omega)$.
\end{rem}
The following lemma is just \cite[Theorem 4.1]{ry99}.

\begin{lem}\label{lem3-22x}
Let $\Omega \subset \R$ be a special Lipschitz domain and  $K$ its associated cone.
There exist four functions $\vz_0,\ \vz,\ \psi_0,\ \psi\in\cs(\rn)$ supported in
$-K$ such that $L_\vz=L_\psi=\fz$ and \eqref{3-22x} holds in $\cd'(\Omega)$
for any $f\in \cs'(\Omega)$.
\end{lem}

\begin{satz}\label{thm3-23x}
Let $\vz_0,\ \vz,\ \psi_0,\ \psi\in\cs(\rn)$ be as in Lemma \ref{lem3-22x}.
Then the map $E $ defined by
\begin{equation}\label{3-23x}
E f :=  \sum_{j=0}^\infty \psi_j\ast (\vz_j\ast f)_{\Omega}\, , \qquad f \in \cs'(\Omega),
\end{equation}
yields a linear and bounded extension operator from $\ce_{u,p,q}^s(\Omega)$
into $\ce_{u,p,q}^s(\rn)$ for all admissible values of $p,\ q,\ u$ and $s$.
\end{satz}

\begin{proof}
The proof is based on that of Theorem \ref{thm3-19x} and similar to that of
\cite[Theorem 4.1(b)]{ry99}. Let $f\in \ce_{u,p,q}^s(\Omega)$. Then $f\in \cs'(\Omega)$ follows,
see Remark \ref{rem3-23x}. By Lemma \ref{lem3-22x} we have
\begin{equation*}
\sum_{j=0}^\infty \psi_j\ast \vz_j \ast f =f \qquad \mbox{in} \quad \cd'(\Omega).
\end{equation*}
Moreover, since the supports of $\psi_0$ and $\psi$ lie in $-K$, it follows that
\[
E (f)_{|_\Omega} = \sum_{j=0}^\infty \psi_j\ast \vz_j\ast f =f.
\]
It remains to prove that the series in \eqref{3-23x} converges in $\cs'(\rn)$
and
\begin{equation}\label{3-23y}
\|\, E(f) \, |{\ce_{u,p,q}^s(\rn)}\| \ls \|\, f \, |{\ce_{u,p,q}^s(\Omega)}\|\, .
\end{equation}
Observe that for any $l,\ j\in\N_0$ and $x\in\rn$ we have
\begin{align*}
\lf|\vz_l\ast\psi_j\ast(\vz_j\ast f)_{\Omega}(x)\r|
&\le \int_\rn |\vz_l\ast\psi_j(z)||(\vz_j\ast f)_{\Omega}(x-z)|\,dz\\
&\le G((\vz_j\ast f)_{\Omega})(x)\int_\rn |\vz_l\ast\psi_j(z)|(1+2^j|z|)^N\,dz\, ,
\end{align*}
where $N$ is chosen as in \eqref{wss-05}.
By \cite[Lemma 2.1]{bpt96}, see also \cite[(4.8)]{ry99}, we know that
for any $M\in \N$ and any $l,\ j\in\N_0$,
$$\int_\rn |\vz_l\ast\psi_j(z)|(1+2^j|z|)^N\,dz
\ls 2^{-|l-j|M}.$$
Thus, there is a $\sigma >0$ such that
\begin{align*}
2^{ls}\lf|\vz_l\ast\psi_j\ast(\vz_j\ast f)_{\Omega}(x)\r|
\ls 2^{-|l-j|\sigma}2^{js}G((\vz_j\ast f)_{\Omega})(x), \qquad x \in \R\, .
\end{align*}
Now by an argument similar to that used in the proof of Theorem \ref{thm3-19x} above,
we conclude that the series in \eqref{3-23x} converges in $\cs'(\rn)$
and that \eqref{3-23y} holds.
\end{proof}

Since multiplication by smooth functions in $\cd(\Omega)$ preserves $\ce_{u,p,q}^s(\rn)$
(see \cite[Theorem 6.1]{ysy}), a standard procedure (see \cite{Stein} or
\cite[p.\,244]{ry99}) allows to reduce the case of a bounded Lipschitz domain to  a special Lipschitz domain.
Now we are in position to formulate the final result of this subsection.

\begin{cor}\label{ryfinal}
Let $\Omega \subset \R $ be  either a  bounded Lipschitz domain if $d\ge 2$
or a bounded interval if $d=1$.
Then there exists  a linear and bounded extension operator $E_\Omega$ such that
\[
E_\Omega \in \cl (\ce_{u,p,q}^s(\Omega)\to \ce_{u,p,q}^s(\R))
\]
simultaneously for  all admissible values of $p,\ q,\ u$ and $s$.
In addition,   for any $f\in\cs'(\Omega)$ we have  $ E_\Omega (f)_{|_\Omega}=f$ in $\cd'(\Omega)$.
\end{cor}

\begin{rem}
 \rm
 A different extension operator for Morrey smoothness spaces has been investigated in   Moura, Neves, Schneider \cite{MNS},
but restricted to a smaller class of domains.
\end{rem}

Let $\Omega \subset \R$ be a bounded domain.
We shall call $\Omega$ an extension domain for ${\ce}^{s}_{u,p,q}(\R)$
if there exists a linear and continuous extension operator
$E \in \cl ({\ce}^{s}_{u,p,q} (\Omega) \to {\ce}^{s}_{u,p,q} (\R))$.


\section{Interpolation of Lizorkin-Triebel-Morrey spaces}\label{s-interpolation}


In this section we will prove our main results. For that purpose we have to deal with complex interpolation. Let $(X_0,X_1)$ be an interpolation couple of Banach spaces.
By $[X_0,X_1]_\Theta$ we denote the result of the complex interpolation of these spaces.
We refer to Calder{\'o}n \cite{ca64}, Bergh, L\"ofstr\"om \cite{BL},
Kre{\u{\i}}n, Petunin, Semenov  \cite{KPS}, Lunardi \cite{lun} and Triebel
\cite{t78} for the basics.
All our investigations will be based on the following essentially  known formulas.

\begin{prop}\label{shestakov}
Let $\tz\in(0,1)$, $s_i\in\rr$,
$p_i \in [1,\infty)$, $q_i\in [1,\fz]$ and $u_i\in[p_i,\fz)$ with $i\in\{0,1\}$. Let $p_0\, u_1=p_1\, u_0$ and $s=(1-\tz)s_0+\tz s_1$ as well as
\[
\frac1p=\frac{1-\tz}{p_0}+\frac\tz{p_1}\, , \quad \frac1q=\frac{1-\tz}{q_0}+\frac\tz{q_1}\, \quad
\mbox{and} \quad \frac1u=\frac{1-\tz}{u_0}+\frac{\tz}{u_1}\, .
\]
{\rm (i)} Then we have
\be\label{wss-06}
 \big[\ce_{u_0,p_0,q_0}^{s_0}(\R),  \ce_{u_1,p_1,q_1}^{s_1}(\R)\big]_\Theta  =  \overline{\ce_{u_0,p_0,q_0}^{s_0}(\R)\cap  \ce_{u_1,p_1,q_1}^{s_1}(\R)}^{\|\, \cdot \, |\ce_{u,p,q}^{s}(\R)\|}.
\ee
{\rm (ii)} Let $\Omega \subset \R$ be either  a bounded Lipschitz domain
if $d\ge 2$ or a bounded interval if $d=1$.
Then
\beqq
 \big[ \ce_{u_0,p_0,q_0}^{s_0}(\Omega),  \ce_{u_1,p_1,q_1}^{s_1}(\Omega)\big]_\Theta
= \overline{\ce_{u_0,p_0,q_0}^{s_0}(\Omega)\cap  \ce_{u_1,p_1,q_1}^{s_1}(\Omega)}^{\|\, \cdot \, |\ce_{u,p,q}^{s}(\Omega)\|}
\eeqq
holds.
\end{prop}

\begin{proof}
Essentially \eqref{wss-06} is proved in \cite{ysy3}. However, the last step, i.e.,
 writing down the explicit formula, has not been done there.
For convenience of the reader we will sketch a proof.
\\
{\em Step 1.}  Proof of (i). We need to switch to the associated sequence spaces
$e_{u_0,p_0,q_0}^{s_0}(\rn)$  based on appropriate wavelet isomorphisms.
We refer to Rosenthal \cite{ro13},  Sawano \cite{sa0}  and Triebel \cite{t14} for more details and proofs. The advantage of these sequence spaces $e^{s}_{u,p,q}(\R)$ compared with the function spaces
$\ce^{s}_{u,p,q}(\R)$
is that they are Banach lattices. Calder{\'o}n
products $X_0^{1-\Theta} \, X_1^\Theta$ are well-defined for Banach lattices,
see Calder{\'o}n \cite{ca64}.
Shestakov \cite{s74,s74m} has proved the following useful identity.
Let $(X_0,X_1)$ be an interpolation couple of Banach lattices and $\Theta \in (0,1)$. Then
\[
[X_0,X_1]_\Theta =
{\overline{X_0\cap X_1}}^{\|\, \cdot \, |X_0^{1-\Theta} \, X_1^\Theta \|}
\, .
\]
Because of
\[
e_{u_0,p_0,q_0}^{s_0}(\rn)^{1-\Theta} \, e_{u_1,p_1,q_1}^{s_1}(\rn)^\Theta =
e_{u,p,q}^{s}(\rn)\, ,
\]
see Yang, Yuan, Zhuo  \cite{yyz}, we find
under the same restrictions as in Proposition \ref{shestakov}
\[
[ e_{u_0,p_0,q_0}^{s_0}(\rn), e_{u_1,p_1,q_1}^{s_1}(\rn)]_\Theta =
{\overline{e_{u_0,p_0,q_0}^{s_0}(\rn) \cap e_{u_1,p_1,q_1}^{s_1}(\rn)}}^{\|\, \cdot \, |e_{u,p,q}^{s}(\rn)\|} \, .
\]
Complex interpolation spaces are invariant under isomorphisms.
Again based on appropriate wavelet isomorphisms
we can turn back to the spaces $\ce_{u,p,q}^{s}(\rn)$. This proves (i).
\\
{\em Step 2.} Proof of (ii). We employ a standard method, see e.g.,
\cite[Thm.~6.4.2]{BL},   \cite[Thm.~1.2.4]{t78} or \cite{t02}.
 Suppose that $E$ is our universal  extension operator with respect to $\Omega$ that was constructed in Corollary \ref{ryfinal}. Then we have
 $E \in \cl (\ce_{u_0,p_0,q_0}^{s_0}(\Omega) \to
 \ce_{u_0,p_0,q_0}^{s_0}(\R))$ and $E \in \cl (  \ce_{u_1,p_1,q_1}^{s_1}(\Omega) \to
 \ce_{u_1,p_1,q_1}^{s_1}(\R)) $ as well as  $E \in \cl (  \ce_{u,p,q}^{s}(\Omega) \to
 \ce_{u,p,q}^{s}(\R)) $.
 It follows that $E$ is a coretraction to the restriction $R$ with respect to $\Omega$.
It is $R\circ E = I$. Here $I$ denotes the identity on the space defined on the domain.
 At the same time $E$ is a linear and continuous extension operator in $\cl (X \to Y)$ where
 \[
 X:=
 \overline{\ce_{u_0,p_0,q_0}^{s_0}(\Omega)\cap  \ce_{u_1,p_1,q_1}^{s_1}(\Omega)}^{\|\, \cdot \, |\ce_{u,p,q}^{s}(\Omega)\|}
 \]
 and
 \[
 Y:=
 \overline{\ce_{u_0,p_0,q_0}^{s_0}(\R)\cap  \ce_{u_1,p_1,q_1}^{s_1}(\R)}^{\|\, \cdot \, |\ce_{u,p,q}^{s}(\R)\|}
 \, .
 \]
 Furthermore, the restriction $R$ applied to $Y$ leads to $X$.
  Hence, Theorem 1.2.4 in \cite{t78} together with Step 1 yield (ii).
   \end{proof}

\begin{rem}
 \rm
 {\rm (i)}
 The formula \eqref{wss-06} itself is explicitely stated in
Hakim, Nogayama, Sawano \cite[Thm.~1.5]{hns19}, but under slightly more restrictive conditions.
Whereas Hakim et al \cite{hns19} reduced  \eqref{wss-06} to results on the
second complex interpolation method of Calder{\'o}n and an abstract result of Bergh \cite{Bergh},
we employed Calder{\'o}n products   and an abstract result of Shestakov \cite{s74,s74m}.
\\
{\rm (ii)}
The interesting formula  \eqref{wss-06} has several forerunners.
It has been used before  in  Lu, Yang, Yuan \cite{lyy}  (restricted to Morrey spaces),
in Sickel, Skrzypczak, Vyb\'\i ral
\cite[4.3]{ssv} (restricted to the classical situation $p=u$) and in Yuan, Sickel, Yang \cite[2.4.3]{ysy3} (general case).
\end{rem}


\subsection{Proof of Theorem \ref{main1}}\label{main1proof}


To prove Theorem \ref{main1} we need the counterpart of Definition \ref{diamond7} for domains. Let $ s \geq 0  $, $ 1 \leq p < u < \infty   $ and $ 1 \leq q \leq \infty   $. We put
\[
 E^s_{u,p,q} (\Omega):= \Big\{f \in  \cd' (\Omega) : ~  \exists \ g \in E^s_{u,p,q} (\R) \quad
 \mbox{such that} \quad f = g \quad \mbox{on} \quad  \Omega \Big\}\, .
\]
It is not difficult to see that we also can write
\[
 E^s_{u,p,q} (\Omega)= \Big\{f \in \ce^s_{u,p,q} (\Omega): ~ D^\alpha f \in \ce^s_{u,p,q} (\Omega)\quad
 \mbox{for all} \quad \alpha \in \N_0^d  \Big\}\, .
\]
But we know even more. There is the following counterpart of Proposition 4.21 in \cite{t08} with almost identical proof.

\begin{lem}\label{hilfe}
 Let $\Omega \subset \R$ be either  a bounded Lipschitz domain
if $d\ge 2$ or a bounded interval if $d=1$.
Let  $ s \geq 0 $, $1\le p < u< \infty$ and $ 1 \leq q \leq \infty $.
Then the set $E^s_{u,p,q} (\Omega)$ is independent of the parameters $s,u,p$ and $q$. Indeed, it holds
\[
 E^s_{u,p,q} (\Omega) = \Big\{f \in C^\infty  (\Omega): ~ D^\alpha f \in L_\infty (\Omega)\quad
 \mbox{for all} \quad \alpha \in \N_0^d  \Big\}\, .
\]
\end{lem}

Let us continue with the proof of the main Theorem \ref{main1}.

\begin{proof}
{\em Step 1.}
Based on Proposition \ref{shestakov} we have to calculate
\[
 \overline{\ce_{u_0,p_0,q_0}^{s_0}(\Omega)\cap  \ce_{u_1,p_1,q_1}^{s_1}(\Omega)}^{\|\, \cdot \, |\ce_{u,p,q}^{s}(\Omega)\|} \, .
 \]
Lemma \ref{hilfe} yields
\[
  E^s_{u,p,q} (\Omega) = E_{u_0,p_0,q_0}^{s_0}(\Omega)\cap  E_{u_1,p_1,q_1}^{s_1}(\Omega).
\]
Therefore just by the definition of the space $\accentset{\diamond}\ce_{u,p,q}^{s}(\Omega) $, Definition \ref{diamond7} and the trivial embeddings $E_{u_i,p_i,q_i}^{s_i}(\Omega) \hookrightarrow
\ce_{u_i,p_i,q_i}^{s_i}(\Omega)$ with $i \in \{0,1\}$ we find
\be
\accentset{\diamond}\ce_{u,p,q}^{s}(\Omega) = \overline{E^s_{u,p,q} (\Omega)}^{\|\, \cdot \, |\ce^s_{u,p,q} (\Omega)\|} \hookrightarrow
\overline{\ce_{u_0,p_0,q_0}^{s_0}(\Omega)\cap  \ce_{u_1,p_1,q_1}^{s_1}(\Omega)}^{\|\, \cdot \, |\ce_{u,p,q}^{s}(\Omega)\|}\, .
\ee
{\em Step 2.} Recall, we assume that either $0 \le s_0 < s_1$ or $0< s_0 = s_1$ and
$q_1\le q_0$, i.e., the conditions of Lemma \ref{step2} are  satisfied.
We claim that
\be\label{wss-09}
 \ce_{u_0,p_0,q_0}^{s_0}(\Omega)\cap  \ce_{u_1,p_1,q_1}^{s_1}(\Omega) \hookrightarrow
 \accentset{\diamond}\ce_{u,p,q}^{s}(\Omega)\, .
\ee
Let $E$ denote the common extension operator.
Let $ f \in \ce_{u_0,p_0,q_0}^{s_0}(\Omega)\cap  \ce_{u_1,p_1,q_1}^{s_1}(\Omega)$. Then
$Ef \in \ce_{u_0,p_0,q_0}^{s_0}(\R)\cap  \ce_{u_1,p_1,q_1}^{s_1}(\R) $.
Let $\psi$ be a function in $\cd (\R)$ such that $\psi (x)= 1$ on $\overline{\Omega}$. Then
the operator $h \mapsto \psi \, \cdot \, h$ belongs to $\cl (\ce_{x,y,z}^{\sigma}(\R) \rightarrow \ce_{x,y,z}^{\sigma}(\R) )$
for all admissible tuples  $(\sigma,x,y,z)$. Hence
$g:= \psi \, \cdot \, Ef \in  \ce_{u_0,p_0,q_0}^{s_0}(\R)\cap  \ce_{u_1,p_1,q_1}^{s_1}(\R)$.
Obviously $(\psi \, \cdot \, Ef)_{|_\Omega} = f$ in $\cd' (\Omega)$.
Let $B$ be a ball such that $\overline{\Omega} \subset \supp \psi \subset B$.
Hence
\[
g \in \ce_{u_0,p_0,q_0}^{s_0}(\R;B)\cap  \ce_{u_1,p_1,q_1}^{s_1}(\R;B) \hookrightarrow
 \accentset{\diamond}\ce_{u,p,q}^{s}(\R)\, ,
\]
see Lemma \ref{step2}. Obviously this means
$f \in \accentset{\diamond}\ce_{u,p,q}^{s}(\Omega)$ and this proves \eqref{wss-09}.
Step 1 and Step 2 combined with Theorem \ref{thm3-23x} prove Theorem \ref{main1}.
\end{proof}

\noindent
{\em Proof of Corollary \ref{MSM}}.
The Corollary is a direct consequence of Theorem \ref{main1} and Lemma \ref{LP}.
{\hspace*{\fill} \qed   \\ }


\subsection{Proof of Theorem \ref{main1b}}


First we recall some well-known embedding relations.
The new restriction (d') guarantees the continuous embedding
\[
\ce_{u_0,p_0,q_0}^{s_0}(\R) \hookrightarrow \ce_{u_1,p_1,q_1}^{t}(\R)
\hookrightarrow \ce_{u_1,p_1,q_1}^{s_1}(\R)
\, ,
\qquad t:= s_0 - d\, \Big(\frac 1{u_0} - \frac 1{u_1}\Big)\, ,
\]
we refer to \cite[Cor.~2.2]{ysy} and \cite{HaSk}.
In addition we get
\[
\ce_{u_0,p_0,q_0}^{s_0}(\R) \hookrightarrow \ce_{u,p,q}^{t_\Theta}(\R)
\hookrightarrow \ce_{u,p,q}^{s}(\R)\, ,
\qquad t_\Theta := s_0 - d\, \Big(\frac 1{u_0} - \frac 1{u}\Big)\, ,
\]
since $p_0 < p < p_1$, $u_0 < u < u_1$, $u_0  \, p =  p_0  \, u$
and
\[
s_1-\frac{d}{u_1}<
 s-\frac{d}{u} = (1-\Theta)\Big(s_0-\frac{d}{u_0}\Big) +
 \Theta \, \Big(s_1-\frac{d}{u_1}\Big)< s_0-\frac{d}{u_0}\, .
\]
Because of $t_\Theta >s$  we may apply Proposition \ref{diamond9} and obtain
\[
\ce_{u_0,p_0,q_0}^{s_0}(\R)
\hookrightarrow \ce_{u,p,q}^{t_\Theta}(\R) \hookrightarrow
\accentset{\diamond}\ce_{u,p,q}^{s}(\R) \, .
\]
Consequently we have
\[
\Big(
 \ce_{u_0,p_0,q_0}^{s_0}(\R) \cap \ce_{u_1,p_1,q_1}^{s_1}(\R)\Big) =
 \ce_{u_0,p_0,q_0}^{s_0}(\R)  \hookrightarrow \accentset{\diamond}\ce_{u,p,q}^{s}(\R) \, .
\]
Hence
\be\label{wss-200}
\overline{\ce_{u_0,p_0,q_0}^{s_0}(\R)\cap  \ce_{u_1,p_1,q_1}^{s_1}(\R)}^{\|\, \cdot \, |\ce_{u,p,q}^{s}(\R)\|}\hookrightarrow \accentset{\diamond}\ce_{u,p,q}^{s}(\R) \, .
\ee
Employing the universal extension operator $E$ from Corollary \ref{ryfinal} we conclude

\[
\overline{\ce_{u_0,p_0,q_0}^{s_0}(\Omega)\cap  \ce_{u_1,p_1,q_1}^{s_1}(\Omega)}^{\|\, \cdot \, |\ce_{u,p,q}^{s}(\Omega)\|}\hookrightarrow \accentset{\diamond}\ce_{u,p,q}^{s}(\Omega) \, .
\]
To prove the reverse embedding we argue as before.
By Lemma \ref{hilfe} we have
\beqq
E_{u,p,q}^{s}(\Omega) & = & \Big(
 E_{u_0,p_0,q_0}^{s_0}(\Omega) \cap E_{u_1,p_1,q_1}^{s_1}(\Omega)\Big)
 \subset  \Big(
 \ce_{u_0,p_0,q_0}^{s_0}(\Omega) \cap \ce_{u_1,p_1,q_1}^{s_1}(\Omega)\Big) ,
\eeqq
which yields
\[
 \accentset{\diamond}\ce_{u,p,q}^{s}(\Omega) \hookrightarrow
  [\ce_{u_0,p_0,q_0}^{s_0}(\Omega), \ce_{u_1,p_1,q_1}^{s_1}(\Omega)]_\Theta \, .
\]
{\hspace*{\fill}  \qed \\ }
{~}\\
\noindent
{\em Proof of Corollary \ref{MSM2}}.
The Corollary is a direct consequence of Theorem \ref{main1b} and Lemma \ref{LP}.
{\hspace*{\fill}  \qed \\ }


\subsection{Proof of Proposition \ref{main3}} \label{mainrestproof}


For the complex method it is well-known that $X_0 \cap X_1$ is a dense subset of
$[X_0,X_1]_\Theta$, see, e.g., \cite[Thm.~4.2.2]{BL} or \cite[Thm.~1.9.3]{t78}. Let the restrictions of Proposition \ref{main3} with respect to
$p_0,p_1, u_0,u_1,q_0,q_1,s_0,s_1 $ and $\Theta $ be satisfied.
The parameters $p,u, q$ and $s$ are then fixed as well.
Without loss of generality we may assume that $\Omega$ contains the ball
$B(0,2)$.
Now we employ Lemma \ref{test3}.
The results immediately carry over to the spaces defined on domains.
Therefore we choose $\alpha:= \frac d{u_0} -s_0$. By assumption
$\alpha >0$ and
$\alpha = \frac d{u_1} -s_1 = \frac d{u} -s$.
Thus Lemma \ref{test3} implies
\[
 f_\alpha \in \ce_{u_0,p_0,q_0}^{s_0}(\Omega)\cap  \ce_{u_1,p_1,q_1}^{s_1}(\Omega)
 \qquad \mbox{and}\qquad f_\alpha \not\in \accentset{\diamond}\ce_{u,p,q}^{s}(\Omega).
\]
This proves the claim.
{\hspace*{\fill}  \qed \\ }


\subsection{Proof of Proposition \ref{main4}}


{\em Step 1.} Proof of (i).
It will be enough to show that there exists a function $ h \in \accentset{\diamond}\ce_{u,p,q}^{s}(\R)    $ such that $ h \not \in  [ \ce_{u_0,p_0,q_0}^{s_0}(\R), \ce_{u_1,p_1,q_1}^{s_1}(\R)]_\Theta $.
Therefore we will work with the family of test functions $h_u$ we investigated  in
Lemma \ref{count_bound1}. Let $ \sigma > 0  $, $ 1 \leq y \leq x < \infty    $ and $ 1 \leq z \leq \infty   $. Then there is the embedding $ \mathcal{E}^{\sigma}_{x,y,z}(\R) \hookrightarrow \mathcal{M}^{x}_{y}(\R)    $. Hence we find
\begin{align*}
\Big(\ce_{u_0,p_0,q_0}^{s_0}(\R) \cap \ce_{u_1,p_1,q_1}^{s_1}(\R)\Big)
\subset \Big(\mathcal{M}^{u_0}_{p_0}(\R) \cap \mathcal{M}^{u_1}_{p_1}(\R)\Big)\, .
\end{align*}
Moreover we observe
\begin{align*}
\overline{\ce_{u_0,p_0,q_0}^{s_0}(\R) \cap \ce_{u_1,p_1,q_1}^{s_1}(\R)}^{\|\, \cdot \, |\ce^s_{u,p,q}(\R)\|} & \subset \overline{\mathcal{M}^{u_0}_{p_0}(\R) \cap \mathcal{M}^{u_1}_{p_1}(\R)}^{\|\, \cdot \, |\ce^s_{u,p,q}(\R)\|} \\
& \subset \overline{\mathcal{M}^{u_0}_{p_0}(\R) \cap \mathcal{M}^{u_1}_{p_1}(\R)}^{\|\, \cdot \, |\mathcal{M}^{u}_{p}(\R)\|}.
\end{align*}
Lemma  \ref{count_bound1} yields $h_u \in \accentset{\diamond}\ce_{u,p,q}^{s}(\R)$. But we have

\[ h_u \not \in  \overline{\mathcal{M}^{u_0}_{p_0}(\R) \cap \mathcal{M}^{u_1}_{p_1}(\R)}^{\|\, \cdot \, |\mathcal{M}^{u}_{p}(\R)\|}     \, .
\]
This has been proved in  \cite{ysy3}, see the proof of
Corollary 2.38, Step 3, page 1891.
\\
{\em Step 2.} Proof of (ii).
This has been proved in \eqref{wss-200}.
{\hspace*{\fill} \qed  \\ }


\section{A few comments to related  results}
\label{overview}


In this section we will collect some more material concerning interpolation of Morrey spaces, smoothness Morrey spaces and their relatives. Let us start with two papers of Lemari{\'e}-Rieusset
\cite{LR,LR2}. Based on earlier work, see
Ruiz, Vega \cite{rv95} and
Blasco, Ruiz, Vega \cite{brv},
he was able to show the importance of the restriction
$u_0  \, p_1 = u_1 \,  p_0$.
Under the restrictions
$1 < p_0 \leq u_0 < \infty$, $1 < p_1 \leq u_1 < \infty$,
$0 < \Theta <1$,  $ \frac1p:=\frac{1-\tz}{p_0}+\frac{\tz}{p_1}$
and $\frac1u:=\frac{1-\tz}{u_0}+\frac{\tz}{u_1}$
he proved that
there exists an interpolation functor F of exponent
$\Theta$ such that $F ( \cm_{p_0}^{u_0}(\R), \cm_{p_1}^{u_1} (\R)) = \cm_p^u(\R)$
if and only if $u_0 \, p_1 = u_1 \,  p_0$. In the meanwhile two interpolation functors are known which have this property, namely
the $\pm$ method of Gustavsson and Peetre and the second complex interpolation method
introduced by Calder{\'o}n. We refer to  Lu, Yang, Yuan \cite{lyy}
and Lemari{\'e}-Rieusset \cite{LR2}, respectively.
Concerning the $\pm$ method, denoted by
$\lf\laz \, \cdot \, , \, \cdot \, ,\tz\r\raz$,
Yuan, Sickel, Yang, \cite{ysy3} have shown
that
$$
\lf\laz \ce_{u_0,p_0,q_0}^{s_0}(\rn), \ce_{u_1,p_1,q_1}^{s_1}(\rn),\tz\r\raz=
\ce_{u,p,q}^{s}(\rn)
$$
holds subject to the restrictions
\begin{itemize}
 \item[(a)] $0 <p_0< p_1<\infty$, $p_0 \le u_0<\infty$, $p_1 \le  u_1 <\infty$;
 \item[(b)] $0 < q_0\, ,  q_1 \le \infty$;
 \item[(c)] $p_0 \, u_1 = p_1 \, u_0$;
 \item[(d)] $s_0,s_1\in  \re $;
  \item[(e)] $0 < \Theta <1$,  $ \frac1p:=\frac{1-\tz}{p_0}+\frac{\tz}{p_1}$,
$\frac1u:=\frac{1-\tz}{u_0}+\frac{\tz}{u_1}$,
$\frac1q:=\frac{1-\tz}{q_0}+\frac{\tz}{q_1}$,
$s:= (1-\tz)s_0 + \tz s_1$.
\end{itemize}
Concerning the  second complex interpolation method, denoted by
$[\, \cdot \, , \, \cdot \, ]^\Theta$,
Hakim, Nagoyama and Sawano \cite{hns19}
proved
$$
[ \ce_{u_0,p_0,q_0}^{s_0}(\rn), \ce_{u_1,p_1,q_1}^{s_1}(\rn)]^\Theta =
\ce_{u,p,q}^{s}(\rn)\, ,
$$
provided that (a)-(e) are satisfied and in addition $p_0,p_1,q_0,q_1 \in (1,\infty)$.
\\
Let us come back to the first complex interpolation method.
Together with the real interpolation method of Lions-Peetre
it is the most important interpolation method. Therefore it is of interest for its own
to understand the spaces
$[ \ce_{u_0,p_0,q_0}^{s_0}(\rn), \ce_{u_1,p_1,q_1}^{s_1}(\rn)]_\Theta$.
In case of the Morrey spaces different
characterizations of $[ \cm^{u_0}_{p_0}(\rn), \cm^{u_1}_{p_1}(\rn)]_\Theta$ can be found in Yuan, Sickel, Yang \cite{ysy3} and Hakim, Nakamura, Sawano \cite{hns17}.
There is a certain number of publications dealing with the interpolation
of subspaces of either  Morrey or of
Lizorkin-Triebel-Morrey spaces.
In particular (but not only),

\begin{itemize}
 \item $\lf\laz \mathring{\ce}_{u_0,p_0,q_0}^{s_0}(\rn), \ce_{u_1,p_1,q_1}^{s_1}(\rn),\tz\r\raz$, $\lf\laz \mathring{\ce}_{u_0,p_0,q_0}^{s_0}(\rn),
\mathring{\ce}_{u_1,p_1,q_1}^{s_1}(\rn),\tz\r\raz$;
\item
$[\mathring{\ce}_{u_0,p_0,q_0}^{s_0}(\rn), \ce_{u_1,p_1,q_1}^{s_1}(\rn)]_\tz$,
$[\mathring{\ce}_{u_0,p_0,q_0}^{s_0}(\rn),
\mathring{\ce}_{u_1,p_1,q_1}^{s_1}(\rn)]_\tz$;
\item
$[\mathring{\ce}_{u_0,p_0,q_0}^{s_0}(\rn), \ce_{u_1,p_1,q_1}^{s_1}(\rn)]^\tz$,
$[\mathring{\ce}_{u_0,p_0,q_0}^{s_0}(\rn),
\mathring{\ce}_{u_1,p_1,q_1}^{s_1}(\rn)]^\tz$;
\item
$[\accentset{\diamond}{\ce}_{u_0,p_0,q_0}^{s_0}(\rn),
\accentset{\diamond}\ce_{u_1,p_1,q_1}^{s_1}(\rn)]^\tz$;
\item
$[\accentset{\diamond}{\ce}_{u_0,p_0,q_0}^{s_0}(\rn),
\accentset{\diamond}{\ce}_{u_1,p_1,q_1}^{s_1}(\rn)]_\tz$
\end{itemize}
and similarly for Morrey spaces, we refer to \cite{lyy}, \cite{yyz}, \cite{ysy3},
\cite{HS16},
\cite{HS17}, \cite{hns17},
\cite{hak18}, \cite{hns19}, \cite{HS20} and \cite{HS202}.
\\
Probably it is of certain interest to notice that the diamond spaces on domains
form a scale under complex interpolation, i.e.,
\be\label{new}
[\accentset{\diamond}{\ce}_{u_0,p_0,q_0}^{s_0}(\Omega),
\accentset{\diamond}{\ce}_{u_1,p_1,q_1}^{s_1}(\Omega)]_\tz =
\accentset{\diamond}{\ce}_{u,p,q}^{s}(\Omega)\, ,
\ee
at least under the restrictions in Theorem
\ref{main1} or in Theorem \ref{main1b}.
This follows from
\[
 E^s_{u,p,q}(\Omega) = E_{u_0,p_0,q_0}^{s_0}(\Omega) \cap  E_{u_1,p_1,q_1}^{s_1}(\Omega)\subset
 [{\ce}_{u_0,p_0,q_0}^{s_0}(\Omega),
{\ce}_{u_1,p_1,q_1}^{s_1}(\Omega)]_\tz =
\accentset{\diamond}{\ce}_{u,p,q}^{s}(\Omega)\, ,
\]
see Lemma \ref{hilfe},
\[
\accentset{\diamond}{\ce}_{u,p,q}^{s}(\Omega) =
  \overline{E^s_{u,p,q}(\Omega)}^{\|\, \cdot \, |{\ce}_{u,p,q}^{s}(\Omega) \|}
 \quad \mbox{and}\quad
  \overline{\accentset{\diamond}{\ce}_{u,p,q}^{s}(\Omega)}^{\|\, \cdot \, |{\ce}_{u,p,q}^{s}(\Omega) \|}=\accentset{\diamond}{\ce}_{u,p,q}^{s}(\Omega)\, .
\]

Let us add a few references to the real method as well.
First results on real interpolation of Besov-Morrey spaces can be found in
Kozono, Yamazaki \cite{KY}.
Mazzucato \cite{ma03} was the first who had dealt with the real interpolation
of Sobolev-Morrey spaces $W^m\cm_p^u (\R)$ and their generalizations to the classes
 $\ce^s_{u,p,2}(\R)$ with $1 <p< u< \infty$. Her result is contained in

\[
\cn^s_{u,p,q}(\R) =  (\ce^{s_0}_{u,p,q_0}(\R), \ce^{s_1}_{u,p,q_1}(\R))_{\theta,q}
\]
if
$s_0,s_1 \in \re$, $s_0 < s_1$, $0< p< u < \infty$, $0 < q_0,q_1,q \le \infty$
and $0 < \Theta < 1$, see \cite{s011a}.
Recently Burenkov, Ghorbanalizadeh, Sawano \cite{BGS} described the $K$-functional for the pair
$(\cm_p^u (a,b), \dot{W}^m\cm_p^u (a,b))$. Here $\dot{W}^m\cm_p^u (a,b)$
refers to the homogeneous Sobolev space.
In \cite{BDN}, \cite{BNC} Burenkov et al studied the real interpolation of slightly modified spaces,
so-called local Morrey spaces. They behave much better under real interpolation than the original Morrey spaces.

Finally, we mention that the interpolation property has been investigated, e.g., in
Adams, Xiao \cite{AX}, Adams \cite{Adams} and
Yuan, Sickel, Yang, \cite{ysy3}, where also further references can be found.


\section{Some open problems}
\label{Ende}


At the end of our paper we would like to address a few open problems which could be of certain interest.

\begin{enumerate}
\item
A general question is about the role of the
Lemari{\'e}-Rieusset condition $u_0\,  p_1 = p_0 \, u_1$.
How do the interpolation spaces look like if this condition is violated?
There are special cases which one should investigate  first like the following.
Let $p_0 = p_1$ and  $u_0 < u_1$.
How do the interpolation spaces
\[
[W^{m_0} \cm_{p_{0}}^{u_0}(\Omega), W^{m_1} \cm_{p_0}^{u_{1}}(\Omega)]_\tz
\]
look like in the case $m_0 < m_1$ ?

\item
What happens if $s_0 -d(\frac{1}{u_0} - \frac{1}{u_1}) < s_1 <s_0$ and
 $u_0\,  p_1 = p_0 \, u_1$ ?
These cases are not treated in the Theorems \ref{main1} and \ref{main1b}.
We refer to the picture at the end of the Introduction.

\item
Find a characterization of $
[\ce^{s_0}_{u_0,p_0,q_0}(\R),\ce^{s_1}_{u_1,p_1,q_1}(\R)]_\tz $ for all admissible constellations of the parameters.
The answer could become technical.

\item We always had to exclude the case $q_0 = q_1 = \infty$.
Under necessary additional restrictions it is known that
\[
 [\ce_{p_0,p_0,\fz}^{s_0}(\R),
\ce_{p_1,p_1,\fz}^{s_1}(\R)]_\Theta =
[F_{p_0,\fz}^{s_0}(\R),F_{p_1,\fz}^{s_1}(\R)]_\Theta= \accentset{\diamond}F_{p_0,\fz}^{s_0}(\R)\,,
\]
see \cite{ssv} and \cite{ysy3}.
So the question is about the characterization of  \\
$[\ce_{u_0,p_0,\fz}^{s_0}(\Omega),
\ce_{u_1,p_1,\fz}^{s_1}(\Omega)]_\Theta\, .$
\item
In contrast to the classical case there are two Besov counterparts of the Lizorkin-Triebel-Morrey spaces,
namely $B_{p,q}^{s,\tau}(\Omega)$ and $\cn_{u,p,q}^{s}(\Omega)$, respectively.
In case of the so-called Besov-Morrey spaces $\cn_{u,p,q}^{s}(\Omega)$
one knows the counterpart of Theorem \ref{main1}, see Theorem 2.45 and Corollary 2.65
in \cite{ysy3}. Let $\Omega \subset \rn$ be a bounded  interval if $d=1$ or a bounded
Lipschitz domain if $d\ge 2$.
Assume that  $0 < p_i \le u_i <\infty$, $s_0,\,s_1 \in \rr$ and   $q_i \in (0,\infty)$, $i\in\{0,1\}$.
Let $s:= (1-\Theta)\, s_0 + \Theta\, s_1$, $\frac 1p : =\frac{1-\tz}{p_0}+ \frac \tz{p_1}$ and $\frac 1q : =\frac{1-\tz}{q_0}+ \frac \tz{q_1}$.
If $u_0 p_1 = u_1 p_0$, then
\[
[ \cn_{u_0,p_0,q_0}^{s_0}(\Omega), \cn_{u_1,p_1,q_1}^{s_1}(\Omega)]_\tz = \accentset{\diamond}{\cn}^{s}_{u,p,q} (\Omega)
\]
holds true for all $\tz \in (0,1)$. There is a surprising difference to the case of the Lizorkin-Triebel-Morrey spaces.
We do not have an influence of the relation between $s_0$ and $s_1$.
The main reason for this more simple behavior can be found in
\[
\accentset{\diamond}{\cn}^{s}_{u,p,q} (\rn) =  {\cn}^{s}_{u,p,q} (\rn) \qquad \mbox{if and only if}
\qquad q\in(0,\infty).
\]
The behavior of the
Besov-type spaces  $B_{p,q}^{s,\tau}(\Omega)$ under complex interpolation seems to be widely open.

\item Let us turn to Corollary \ref{MSM}. Obviously the case  $p_0=1$ has been left out.
What happens if $p_0=1$ ?

\item Probably even more difficult is the question around the use of the
extension property of our function spaces on domains.
Is there a wider class of domains than bounded Lipschitz domains allowing the validity of Theorem \ref{main1} ?

\item
We concentrated on Banach spaces in our paper. There is a well developed theory
of the function spaces also for
values $u,p,q\in (0,1)$, see \cite{ysy}, \cite{s011,s011a} and \cite{t14}.
Extensions of the complex method to quasi-Banach spaces are known as well, we refer to
\cite{k86b}, \cite{km98}, \cite{kmm} and \cite{y14}.
\end{enumerate}


\noindent\textbf{Acknowledgements}\quad

Ciqiang Zhuo is supported by the Construct Program of the Key Discipline in
Hunan Province, the National Natural Science Foundation of China (Grant Nos. 11701174, 11831007, 11871100 and China Scholarship Council (Grant No. 201906725036).

Marc Hovemann is funded by a Landesgraduiertenstipendium which is a scholarship from the Friedrich-Schiller university and the Free State of Thuringia.


\end{document}